\documentclass[journal]{IEEEtran}
\usepackage{mathrsfs}
\usepackage{graphicx}
\usepackage{latexsym}
\usepackage{ulem}
\usepackage{algorithmic}
\usepackage{subfigure}
\usepackage{amsmath}
\usepackage{cite}
\usepackage{amssymb}
\usepackage{multirow}
\usepackage{amsthm}
\usepackage{color}
\usepackage{diagbox}
\usepackage{picins}
\usepackage{wrapfig}
\usepackage[ruled,linesnumbered]{algorithm2e}
\usepackage[colorlinks,linkcolor=blue,anchorcolor=blue,citecolor=black]{hyperref}

\newtheorem{theorem}{Theorem}
\newtheorem{corollary}{Corollary}
\newtheorem{assumption}{Assumption}
\newtheorem{remark}{Remark}
\newtheorem{lemma}{Lemma}

\begin{document}
\normalem
\title{MUSIC: Accelerated Convergence for Distributed Optimization With Inexact and Exact Methods}
\author{Mou Wu, \emph{Member, IEEE}, Haibin Liao, Zhengtao Ding, \emph{Senior Member, IEEE}, Yonggang Xiao
\thanks{Mou Wu and Yonggang Xiao are with School of Computer Science and Technology, Hubei University of Science and Technology, Xianning 437100, PR China. (Email: mou.wu@163.com; xygone@gmail.com)}
\thanks{Haibin Liao is with School of Electronic and Electrical Engineering, Wuhan Textile University, Wuhan, China. (Email: liao\_haibing@163.com)}
\thanks{Zhengtao Ding is with the Department of Electrical and Electronic Engineering, The University of Manchester, Manchester M13 9PL, U.K. (E-mail: zhengtao.ding@manchester.ac.uk)}
}

\maketitle
\begin{abstract}
Gradient-type distributed optimization methods have blossomed into one of the most important tools for solving a minimization learning task over a networked agent system. However, only one gradient update per iteration is difficult to achieve a substantive acceleration of convergence. In this paper, we propose an accelerated framework named as MUSIC allowing each agent to perform multiple local updates and a single combination in each iteration. More importantly, we equip inexact and exact distributed optimization methods into this framework, thereby developing two new  algorithms that exhibit accelerated linear convergence and high communication efficiency. Our rigorous convergence analysis reveals  the sources of steady-state errors arising from inexact policies and offers effective solutions. Numerical results based on synthetic and real datasets demonstrate both our theoretical motivations and analysis, as well as performance advantages.

\end{abstract}

\normalem
\section{Introduction and Motivation}

\IEEEPARstart{D}{istributed} computation for minimizing a sum of convex functions has been motivated by wide applications in engineering and technological domains including sensor and robot networks \cite{jaleel2020distributed}, smart grid\cite{molzahn2017survey}, large scale machine learning \cite{nedic2020distributed} and neural networks \cite{yang2017collaborative}. Instead of seeking a centralized solution, many distributed optimization methods have been proposed to address  such problems. The popular first-order gradient methods include the Distributed Gradient Descent (DGD) \cite{nedic2009distributed,nedic2016stochastic,7963560,nedic2020distributed}, Distributed Nesterov Gradient \cite{li2020distributed,lu2020nesterov}, and Distributed Gradient Tracking \cite{koloskova2021improved,pu2021distributed}. Without exception, the successful implementation of these algorithms depends on two critical steps, i.e., local computations based on a local objective function and input data, and local communications based on information exchange with their immediate neighbors over the underlying network.

By reordering the update and combination steps,  a variant of DGD \eqref{eq3}-\eqref{eq4} structure is the  diffusion-based Adapt-Then-Combine (ATC) method. The asynchrony of two processes in the context of  distributed gradient projection is studied in \cite{9773308}.  First-order DGD/ATC methods  offer significant advantages, including low computational costs and rapid convergence. However, these methods inherit their inexact nature. Specifically, they do not converge to the exact minimizer denoted as $\textbf{x}^{*}$ but rather approach the $O(\alpha)$ or $O(\alpha^{2})$-neighborhood around $\textbf{x}^{*}$ when  a fixed step size $\alpha$ is employed \cite{nedic2009distributed}. Such a steady-state bias leads to inexact convergence. In other words, exact convergence  can be achieved by using a diminishing step size. However, the resulting slow convergence rate becomes unacceptable, both in theoretical and practical context. Therefore, in an inexact setting, a dilemma arises when there is a simultaneous requirement for both accuracy and speed.

Communication cost is another important consideration when designing distributed optimization methods over a networked learning system. Recently, originating from a centralized optimization with some specified considerations, such as data heterogeneity and partial device participation, federated learning \cite{kairouz2021advances,yang2019federated,konevcny2016federated,li2020federated} has introduced  a novel computing paradigm for machine learning in which each agent is allowed to perform multiple local updates before communicating with other neighboring agents.  The resulting benefits  include less communication and faster convergence.

While multiple updates strategy achieves success in the emerging federated learning, it is not clear whether it can provide workable solutions in the distributed optimization setting. Inspired by this motivation, to reach the aim of less communication cost while accelerating convergence, in this paper, we propose the \underline{M}ulti-\underline{U}pdates \underline{SI}ngle-\underline{C}ombination (referred to as MUSIC) framework designed for two gradient-type (i.e., ATC and exact diffusion) distributed optimization methods with inexact and exact estimations to satisfy different accuracy requirements.

\subsection{Related Work}

First-order gradient-based optimization learning methods can be informally classified into three distinct classes: inexact, non-accelerated exact and accelerated exact algorithms.

\textbf{Inexact methods.} Inexact first-order optimization has been studied intensely and a great deal of research works are carried out, such as the well-known DGD method \cite{nedic2009distributed,yuan2016convergence} for undirected  networks, the (sub)gradient-push methods \cite{nedic2014distributed,nedic2016stochastic} for directed  networks and ATC/CTA (Combine-Then-Adapt) for diffusion networks \cite{sayed2014adaptation,sayed2014adaptive}. The corresponding asynchronous and stochastic versions with convergence rate analysis are proposed in \cite{5585721} and \cite{sundhar2010distributed}, respectively. Similar to ATC/CTA, learning-then-consensus (LTC) and consensus-then-learning (CTL) algorithms are proposed with stochastic gradient noises \cite{9406173}. Compared with the centralized gradient method, the distributed methods incurs slower convergence rate. However, in present-day society with extra attention on data privacy, collecting all data in a centralized machine is often unrealistic. Benefiting from low computational cost and algorithmic simplicity, these inexact gradient-type methods have proven to be fundamental and extremely popular. Therefore, they are highly recommended when high precision is not required. On the other hand, achieving faster convergence while maintaining almost the same precision as the existing inexact methods remains an open question.  The heavy-ball and Nesterov's momentum accelerations are studied in an inexact way \cite{9286906}.

\textbf{Non-accelerated exact methods.} Numerous bias-correction methods with a fixed step size have been proposed to address the dilemma between convergence accuracy and speed in the context of inexact solutions. The well-studied \underline{EX}act firs\underline{T}-orde\underline{R} \underline{A}lgorithm (EXTRA) \cite{shi2015extra,li2020revisiting} uses the gradients of last two iterates to address the bias problem in a consensus way as  DGD. A gradient tracking algorithm with variance reduction (GT-VR) is proposed to solve large-scale non-convex finite-sum optimization \cite{9772052}. Instead of exchanging  the estimations from previous two local updates, the \underline{N}etwork \underline{I}n\underline{D}ependent \underline{S}tep-size (NIDS) \cite{li2019decentralized} exchanges the gradient adapted estimations. Different from the EXTRA, the gradient-tracking methods \cite{sun2022distributed,li2020communication}  use current gradient information to track the averaged gradients of the overall objective. The \underline{D}istributed
\underline{I}nexact \underline{G}radient and gradient-track\underline{ing} (DIGing) \cite{nedic2017achieving,sun2022distributed,li2020communication} applies the gradient-tracking technique in  time varying graphs. To obtain better effect of bias-correction, these methods interact more frequently with neighborhood than inexact ones, thereby resulting in more expensive communication. Motivated by the fact that traditional diffusion strategies outperform traditional consensus strategies \cite{sayed2014adaptation,li2022consensus}, exact diffusion \cite{yuan2018exact} is proposed to correct the bias by removing the difference between local and global estimates from the previous iteration. Convergence analysis \cite{yuan2018exact2} shows that exact diffusion has a wider stability range with faster convergence rate than the EXTRA. Influence of bias-correction on distributed stochastic setting is studied in \cite{yuan2020influence}. Nested Exact Alternating Recursion-DGD (NEAR-DGD) \cite{berahas2018balancing} can  converge to an exact consensual solution by balancing communication and computation, but communication amounts is huge to reach this goal.

\textbf{Accelerated exact methods.} The accelerated versions of some exact methods are proposed in \cite{li2020revisiting,li2021accelerated,8812696,kovalev2020optimal}. However, these accelerated methods typically require a meticulous selection of  numerous parameters including the step sizes, and a comprehensive understanding of global knowledge. For example, in the Accelerated Extra \cite{li2020revisiting}, the parameters including the second largest singular value of the combination matrix, convex and smooth coefficients of objective functions must be estimated in advance. Both ACCelerated Gradient Tracking (Acc-GT) \cite{li2021accelerated} and ACCelerated Distributed Nesterov Gradient Descent (Acc-DNGD) \cite{8812696} use four intermediate variables to facilitate  three information exchanges per iteration. Accelerated Proximal Alternating Predictor-Corrector (APAPC) \cite{kovalev2020optimal} requires only one information exchange, however, four auxiliary parameters including the step size need to be set with complex calculations.

\textbf{Multiple updates structure.} The idea of multiple updates in fact is not proposed firstly in  federated learning. One can track the seminal work in the centralized stochastic gradient descent (SGD)  known as local update SGD \cite{mangasarian1995parallel}, which shows faster convergence and less communication by multiple local updating.  Its recent variants \cite{stich2018local,wang2021cooperative,khaled2020tighter} (e.g., local SGD, Periodic Simple-Averaging SGD (PSASGD), Elastic Averaging SGD (EASGD), and decentralized parallel SGD) benefit from such a promising idea that allow workers to perform multiple local updates to the model and then combine the local models periodically. Notably, the well-known federated averaging (FedAvg) algorithm \cite{kairouz2021advances} is a derivative of local SGD, specifically designed for unbalanced participating devices.

\subsection{Contributions and organization}
Theoretically and experimentally, it is confirmed that our method enhances the distributed EASGD method (please see (\ref{eq86}) (\ref{eq87}) to an obvious higher level with better performance. To the best of our knowledge, the proposed local correction technique has never been reported in literature.

Our main contributions and novelties are summarized as follows.
\begin{itemize}
                                  \item To the best of our knowledge, this work is the first one to implement Multi-Updates SIngle-Combination (MUSIC) strategy for solving distributed deterministic optimization problems. As a result, numerous state-of-the-art methods (e.g., exact and inexact, accelerated and non-accelerated, first-order and second-order) can potentially employ such structure to obtain performance improvements due to its visible benefits.
                                  \item Furthermore, the MUSIC-based novel local correction technique noticeably improves the reduction of the error neighborhood size. Both theoretically and experimentally, we have confirmed that our method significantly elevates the performance of the distributed EASGD method (please see \eqref{eq86} and \eqref{eq87}).
                                  \item Moreover, our analysis method provides an intuitive  and rigorous theoretical understanding of how the convergence of MUSIC evolves asymptotically  and its steady-state error compositions. And most particularly, the proof structure is established in a seamless way from inexact MUSIC to exact MUSIC, resulting in a clear performance comparison.
                                      \item Finally, compared to existing methods, whether exact or accelerated, our proposed Exact MUSIC method is simpler yet more effective in terms of acceleration capability, while also offering the best communication complexity. This assertion is substantiated by both theoretical analysis and experimental results.
                                \end{itemize}

The paper is organized as follows. Section \ref{II} reviews relevant preliminaries. The inexact and exact MUSIC methods with convergence analysis and numerical experiments are presented in Section \ref{III} and \ref{IV}, respectively. Section \ref{V} concludes the paper and proposes future work.

\subsection{Notations}
For a better understanding for this work, throughout the paper the involved matrices and vectors are denoted with capital letters and small letters in bold, respectively, while the scalars are denoted in normal font. Specially, $\textbf{x}^{T}$ means the transpose of vector $\textbf{x}$. The operator $\otimes$  denotes the Kronecker product. $\|\cdot\|$ denotes the Euclidean norm of vectors and the spectral norm of matrices. $\lfloor x \rfloor$ denotes the greatest integer not exceeding $x$. The inner product in the Euclidean space is denoted by $\langle\cdot\rangle$. We use the subscript (e.g., $i,j$) and superscript ($t$) to denote the agent and time indexes, respectively.

\section{Preliminaries}\label{II}

In this section,  we briefly review the classical first-order DGD and ATC methods. The target of distributed optimization is to minimize a finite-sum loss of all agents as follows:
\begin{equation}  \label{eq1}
\begin{aligned}
\textbf{x}^{*}=\textrm{arg}\min\limits_{\textbf{x}\in\mathbb{R}^{p}}\sum\limits_{i=1}^{N}f_{i}(\textbf{x}),
\end{aligned}
\end{equation}
where $f_{i}$ is the local objective function held by agent $i$ over a networked system and assumed to be $\mu-$strongly convex  and $L$-smooth. It is noted that the local objective function $f_{i}$ may have different local minimizers denoted by $f_{i}^{*}$ due to the constraint that every agent has different neighborhoods and local datasets. In such distributed topologies, all agents seek to obtain the global solution $\textbf{x}^{*}$ by working cooperatively.

The DGD method for solving (\ref{eq1}) takes the following form:
\begin{equation}  \label{eq2}
\textbf{x}^{t+1}_{i}=\sum\limits_{j\in\mathcal{N}_{i}}w_{ij}\textbf{x}^{t}_{i}-\alpha \nabla f_{i}(\textbf{x}^{t}_{i}),
\end{equation}
where $\nabla$ is the gradient operator, $\textbf{x}^{t}_{i}$ denotes the estimate of an arbitrary agent $i$ at iteration $t$.  The weight $w_{ij}$ held by agent $i$ is used to scale the data that flows from agent $j$ to $i$ with the basic constraints $\sum_{j\in\mathcal{N}_{i}}w_{ij}=1$ and $w_{ij}\geq0$ for any $i$, where $\mathcal{N}_{i}$ is the neighboring set of agent $i$ including itself. Moreover, it is necessary that $w_{ij}=0$ for non-adjacent agents $j\notin \mathcal{N}_{i}$. The formulation \eqref{eq2} can be rewritten in two steps, i.e.,

\begin{align}
\label{eq3} &\textbf{v}^{t}_{i}=\sum\limits_{j\in\mathcal{N}_{i}}w_{ij}\textbf{x}^{t}_{i}, \;\;\;\;\;\textbf{(combine)}\\
\label{eq4} &\textbf{x}^{t+1}_{i}=\textbf{v}^{t}_{i}-\alpha \nabla f_{i}(\textbf{x}^{t}_{i}),\;\;\;\;\;\textbf{(local update)}
\end{align}
where $\textbf{v}^{t}_{i}$ is the aggregated estimate by receiving synchronous estimates from other neighboring agents, while $\textbf{x}^{t}_{i}$ is a local estimate for $\textbf{x}^{*}$ by using a modified gradient descent method. Note that we use the constant step size $\alpha$ for all agents during iterations.

Different from the DGD, the ATC method carries out the following iteration
\begin{align}
\label{eq5} &\textbf{v}^{t+1}_{i}=\textbf{x}^{t}_{i}-\alpha \nabla f_{i}(\textbf{x}^{t}_{i}),\;\;\;\textbf{(local update)}\\
\label{eq6} &\textbf{x}^{t+1}_{i}=\sum\limits_{j\in\mathcal{N}_{i}}w_{ij}\textbf{v}^{t+1}_{i}. \;\;\;\;\;\textbf{(combine)}
\end{align}
Aside from the obvious difference of execution order between \eqref{eq3}-\eqref{eq4} and \eqref{eq5}-\eqref{eq6}, ATC  employs the traditional gradient descent rather than the modified one as \eqref{eq4}. This particular implementation has demonstrated improved precision, with the same level of communication overhead as the DGD \cite{sayed2014adaptation,sayed2014adaptive}.  This advantage arises from incorporating the latest estimates in the gradient computation.

\section{Inexact MUSIC}\label{III}

In this section, we propose the inexact MUSIC, which is a combination  of the inexact ATC and the MUSIC framework. We show that inexact MUSIC exhibits a linear convergence rate faster than that of ATC.

\textbf{Algorithm description.}  Intuitively, the inexact MUSIC algorithm consists of two loop iterations, i.e., intra-agent computation loop and inter-agent communication loop. Here, we denote the total number of iterations as $T$ during the algorithm  with one combination  occurring every $E$ local update steps. Since a communication only occurs during the combination step, the number of communication rounds for each agent is equal to  $\lfloor T/E\rfloor$, where one round means that an agent $i$ sends the current estimate $\textbf{x}_{i}^{t}$ to its neighboring agents and receives $\textbf{x}_{j\in\mathcal{N}_{i}}^{t}$ from them.

Let $\mathcal{I}_{E}$ be the set of combination steps, i.e., $\mathcal{I}_{E}=\{kE|k=0,1,2,\ldots,\lfloor T/E\rfloor\}$. Therefore, there exists a time $t^{0}=kE\leq t$ for any $t>0$ satisfying $t-t^{0}\leq E$. We can describe the inexact MUSIC with the following iteration
\begin{equation}  \label{eq7}
\begin{aligned}
\textbf{x}_{i}^{t+1}=\begin{cases}
\textbf{x}_{i}^{t}-\alpha \nabla f_{i}(\textbf{x}^{t}_{i})  &\textrm{if}\;\; t+1\notin \mathcal{I}_{E} \\
\sum\limits_{j\in \mathcal{N}_{i}}w_{ij}(\textbf{x}_{j}^{t}-\alpha \nabla f_{j}(\textbf{x}^{t}_{j}))  &\textrm{if}\;\; t+1\in \mathcal{I}_{E}
\end{cases},
\end{aligned}
\end{equation}
where the weights $w_{ij}$ is same as in \eqref{eq3} and \eqref{eq6}. The resulting weight matrix $\textbf{W}$ with entry $w_{ij}$ and size $N\times N$ is a doubly stochastic
matrix, i.e., it has non-negative entries and satisfies $\textbf{W}\mathrm{\textbf{1}}_{N}=\mathrm{\textbf{1}}_{N}$ and $\textbf{W}^{T}\mathrm{\textbf{1}}_{N}=\mathrm{\textbf{1}}_{N}$, where $\mathrm{\textbf{1}}_{N}$ is a column vector of size $N$ with all its entries equal to one. Alternative choices of $\textbf{W}$ include Laplacian rule, Metropolis rule and Maximum degree rule \cite{sayed2014adaptation,sayed2014adaptive}. If we introduce the  intermediate variable $\textbf{v}_{i}^{t}$ as \eqref{eq5}, \eqref{eq7} can be rewritten by

\begin{align}
\label{eq8} &\textbf{v}_{i}^{t+1}=\textbf{x}_{i}^{t}-\alpha \nabla f_{i}(\textbf{x}^{t}_{i}),\\
\label{eq9} &\textbf{x}_{i}^{t+1}=\begin{cases}
\textbf{v}_{i}^{t+1}  &\textrm{if}\; t+1\notin \mathcal{I}_{E} \\
\sum\limits_{j\in \mathcal{N}_{i}}w_{ij}\textbf{v}_{j}^{t+1}  & \textrm{if}\;t+1\in \mathcal{I}_{E}
\end{cases},
\end{align}
which are useful for subsequent convergence analysis. In \eqref{eq8}, $\textbf{v}_{i}^{t+1}$ is a single result of gradient descent at $\textbf{x}^{t}_{i}$. It is noted that the inexact MUSIC reduces to the ATC method when $E=1$.  Fig. \ref{fig1} illustrates  the workflow of inexact MUSIC. During the inner (local update) iterations, each agent performs $E$ gradient descent steps as defined in \eqref{eq8}. In the outer (combination) iterations, each agent aggregates estimates from its neighborhood using weighted consensus as described in \eqref{eq9}. When the terminal conditions, such as the expected level of accuracy or the designated number of iterations, are met, the algorithm comes to a halt.

\begin{figure}[htbp]
\centering
\centerline{\includegraphics[width=5.5cm,height=7cm]{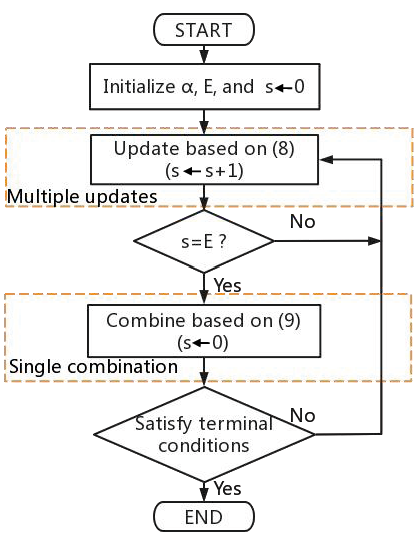}}
\caption{Illustration of workflow in the inexact MUSIC. Note that a temporary variable $s$ is used to control the quantity of local updates.}\label{fig1}
\end{figure}

\subsection{Convergence analysis}\label{IIA}
Before jumping to the convergence analysis, we first introduce the following common assumptions for convex and smooth functions.
\subsubsection{Assumptions and additional notations}

\begin{assumption}\label{assum1}
Local objective function $f_{i}$ is $\mu$-strongly convex:
\begin{equation}  \label{eq10}
f_{i}(\textbf{x})\geq f_{i}(\widehat{\textbf{x}})+(\textbf{x}-\widehat{\textbf{x}})^{T}\nabla f_{i}(\widehat{\textbf{x}})+\frac{\mu}{2}\|\textbf{x}-\widehat{\textbf{x}}\|^{2},
\end{equation}
for any $\textbf{x}$ and $\widehat{\textbf{x}}\in\mathbb{R}^{p}$. Accordingly, it follows from the above
\begin{equation}  \label{eq10-1}
\|\nabla f_{i}(\textbf{x})\|^{2}\geq2\mu(f_{i}(\textbf{x})-f_{i}^{*}).
\end{equation}
\end{assumption}

\begin{assumption} \label{assum2}Local objective function $f_{i}$ is L-smooth:
\begin{equation}  \label{eq11}
f_{i}(\textbf{x})\leq f_{i}(\widehat{\textbf{x}})+(\textbf{x}-\widehat{\textbf{x}})^{T}\nabla f_{i}(\widehat{\textbf{x}})+\frac{L}{2}\|\textbf{x}-\widehat{\textbf{x}}\|^{2},
\end{equation}
for any $\textbf{x}$ and $\widehat{\textbf{x}}\in\mathbb{R}^{p}$. Accordingly, it follows from the above
\begin{equation}  \label{eq12}
\|\nabla f_{i}(\textbf{x})\|^{2}\leq2L(f_{i}(\textbf{x})-f_{i}^{*}).
\end{equation}
\end{assumption}

\begin{assumption}\label{assum3}
Based on \eqref{eq10-1} and \eqref{eq12}, the gradients for $f_{i}$ is bounded: $0\leq G_{min}\leq\|\nabla f_{i}(\textbf{x}_{i}^{t})\|\leq G_{max}$  for all $i=1,\ldots,N$ and $t=1,\ldots,T$.
\end{assumption}

Assumptions \ref{assum1} and \ref{assum2} are generally applicable when the local objective function $f_{i}$ is $\mu$-strongly convex and L-smooth \cite{nesterov2018lectures,8812696}. Assumption \ref{assum3} on bounded gradients is a common requirement in numerous distributed optimization results \cite{xi2016distributed,simonetto2017decentralized,simonetto2016class}. For notational convenience, we introduce the following quantities that are used in our analysis:
$$\overline{\textbf{v}}_{i}^{t}=\sum\limits_{j=1}^{N}w_{ij}\textbf{v}_{j}^{t},\;\;\;\overline{\textbf{x}}_{i}^{t}=\sum\limits_{j=1}^{N}w_{ij}\textbf{x}_{j}^{t},$$
and
$$\overline{\textbf{g}}_{i}^{t}=\sum\limits_{j=1}^{N}w_{ij}\nabla f_{j}(\textbf{x}_{j}^{t}).$$

\subsubsection{Key lemmas}
Here, we present several key lemmas in order to establish the general dynamical system related to network optimality gap $\|\overline{\textbf{v}}_{i}^{t}-\textbf{x}^{*}\|^{2}$. Firstly, we obtain the bounded result of one step gradient descent \eqref{eq8}, which provides an important relation for later use.

\begin{lemma}\label{lem1}\textbf{(One step gradient descent)}
Under Assumptions \ref{assum1} and \ref{assum2}, if the step size $\alpha$ satisfies $\alpha\leq\frac{1}{2L}$ for one step gradient descent \eqref{eq8} of the inexact MUSIC \eqref{eq8}-\eqref{eq9}, we have
\begin{equation}  \label{eq13}
\begin{aligned}
\parallel\overline{\textbf{v}}_{i}^{t+1}-\textbf{x}^{*}\parallel^{2}\leq&(1-\mu\alpha)\|\overline{\textbf{x}}_{i}^{t}-\textbf{x}^{*}\|^{2}+\sum\limits_{j=1}^{N}w_{ij}\|\textbf{x}_{j}^{t}-\overline{\textbf{x}}_{i}^{t}\|^{2}\\
&+\gamma\sum\limits_{j=1}^{N}w_{ij}\|\textbf{x}_{j}^{t}-\overline{\textbf{x}}_{j}^{t}\|^{2}+2\alpha\tau,
\end{aligned}
\end{equation}
where $f_{i}(\textbf{x}^{*})-f_{i}^{*}\leq\tau$ for any agent $i$, $\gamma=\frac{\alpha(1-2L\alpha)}{\pi}$ and $0<\pi<\frac{1}{L}$.
\end{lemma}

\begin{proof}

Based on \eqref{eq8} and the definitions of $\overline{\textbf{v}}_{i}^{t}$, $\overline{\textbf{x}}_{i}^{t}$ and $\overline{\textbf{g}}_{i}^{t}$, we have
\begin{equation}  \label{eq14}
\begin{aligned}
\parallel\overline{\textbf{v}}_{i}^{t+1}-&\textbf{x}^{*}\parallel^{2}=\parallel\overline{\textbf{x}}_{i}^{t}-\alpha\overline{\textbf{g}}_{i}^{t}-\textbf{x}^{*}\parallel^{2}\\
&=\parallel\overline{\textbf{x}}_{i}^{t}-\textbf{x}^{*}\parallel^{2}\underbrace{-2\alpha\langle\overline{\textbf{x}}_{i}^{t}-\textbf{x}^{*},\overline{\textbf{g}}_{i}^{t}\rangle}\limits_{A_{1}}+\underbrace{\alpha^{2}\parallel\overline{\textbf{g}}_{i}^{t}\parallel^{2}}\limits_{A_{2}}.
\end{aligned}
\end{equation}

We first bound $A_{2}$ as
\begin{equation}  \label{eq15}
\begin{aligned}
A_{2}&=\alpha^{2}\parallel\overline{\textbf{g}}_{i}^{t}\parallel^{2}=\alpha^{2}\bigg\|\sum\limits_{j=1}^{N}w_{ij}\nabla f_{j}(\textbf{x}_{j}^{t})\bigg\|^{2}\\
&\leq\alpha^{2}\sum\limits_{j=1}^{N}w_{ij}\|\nabla f_{j}(\textbf{x}_{j}^{t})\|^{2}\\
&\leq2L\alpha^{2}\sum\limits_{j=1}^{N}w_{ij}(f_{j}(\textbf{x}_{j}^{t})-f_{j}^{*}),
\end{aligned}
\end{equation}
where the first inequality arises from the convexity of $f_{j}$, the last inequality is based on the $L$-smoothness of $f_{j}$.

To bound $A_{1}$, we have
\begin{equation}  \label{eq16}
\begin{aligned}
A_{1}&=-2\alpha\langle\overline{\textbf{x}}_{i}^{t}-\textbf{x}^{*},\overline{\textbf{g}}_{i}^{t}\rangle\\
&=-2\alpha\langle\overline{\textbf{x}}_{i}^{t}-\textbf{x}^{*},\sum\limits_{j=1}^{N}w_{ij}\textbf{g}_{j}^{t}\rangle\\
&=-2\alpha\sum\limits_{j=1}^{N}w_{ij}\langle\overline{\textbf{x}}_{i}^{t}-\textbf{x}^{*},\textbf{g}_{j}^{t}\rangle\\
&=-2\alpha\sum\limits_{j=1}^{N}w_{ij}\langle\overline{\textbf{x}}_{i}^{t}-\textbf{x}_{j}^{t}+\textbf{x}_{j}^{t}-\textbf{x}^{*},\textbf{g}_{j}^{t}\rangle\\
&=-2\alpha\sum\limits_{j=1}^{N}w_{ij}\langle\overline{\textbf{x}}_{i}^{t}-\textbf{x}_{j}^{t},\textbf{g}_{j}^{t}\rangle-2\alpha\sum\limits_{j=1}^{N}w_{ij}\langle\textbf{x}_{j}^{t}-\textbf{x}^{*},\textbf{g}_{j}^{t}\rangle,
\end{aligned}
\end{equation}
where we use $\textbf{g}_{j}^{t}\triangleq\nabla f_{j}(\textbf{x}_{j}^{t})$ in the second equality.

By $\mu$-strong convexity, we have
\begin{equation}  \label{eq17}
\begin{aligned}
-\langle\textbf{x}_{j}^{t}-\textbf{x}^{*},\textbf{g}_{j}^{t}\rangle\leq -(f_{j}(\textbf{x}_{j}^{t})-f_{j}(\textbf{x}^{*}))-\frac{\mu}{2}\|\textbf{x}_{j}^{t}-\textbf{x}^{*}\|^{2}.
\end{aligned}
\end{equation}
By AM-GM inequality, it is known that $\pm2\langle\textbf{a},\textbf{b}\rangle\leq\alpha\|\textbf{a}\|^{2}+\alpha^{-1}\|\textbf{b}\|^{2}$ for any vectors $\textbf{a}$ and $\textbf{b}$. Thus, we have
\begin{equation}  \label{eq18}
\begin{aligned}
-2\langle\overline{\textbf{x}}_{i}^{t}-\textbf{x}_{j}^{t},\textbf{g}_{j}^{t}\rangle\leq\alpha^{-1}\|\overline{\textbf{x}}_{i}^{t}-\textbf{x}_{j}^{t}\|^{2}+\alpha \|\textbf{g}_{j}^{t}\|^{2}.
\end{aligned}
\end{equation}

Substituting \eqref{eq17} and \eqref{eq18} into \eqref{eq16}, it follows that
\begin{equation}  \label{eq19}
\begin{aligned}
&A_{1}+A_{2}\leq A_{2}+\alpha\sum\limits_{j=1}^{N}w_{ij}\bigg(\frac{1}{\alpha}\|\overline{\textbf{x}}_{i}^{t}-\textbf{x}_{j}^{t}\|^{2}+\alpha\|\textbf{g}_{j}^{t}\|^{2}\bigg)\\
&\;\;-2\alpha\sum\limits_{j=1}^{N}w_{ij}\bigg(f_{j}(\textbf{x}_{j}^{t})-f_{j}(\textbf{x}^{*})+\frac{\mu}{2}\|\textbf{x}_{j}^{t}-\textbf{x}^{*}\|^{2}\bigg)\\
&\leq-\mu\alpha\|\overline{\textbf{x}}_{i}^{t}-\textbf{x}^{*}\|^{2}+\sum\limits_{j=1}^{N}w_{ij}\|\overline{\textbf{x}}_{i}^{t}-\textbf{x}_{j}^{t}\|^{2}\\
&+\underbrace{2\alpha\sum\limits_{j=1}^{N}w_{ij}\big[2L\alpha(f_{j}(\textbf{x}_{j}^{t})-f_{j}^{*})-(f_{j}(\textbf{x}_{j}^{t})-f_{j}(\textbf{x}^{*}))\big]}\limits_{B},
\end{aligned}
\end{equation}
where we use the fact of $\|\overline{\textbf{x}}_{i}^{t}-\textbf{x}^{*}\|^{2}=\bigg\|\sum\limits_{j=1}^{N}w_{ij}(\textbf{x}_{j}^{t}-\textbf{x}^{*})\bigg\|^{2}\leq\sum\limits_{j=1}^{N}w_{ij}\|\textbf{x}_{j}^{t}-\textbf{x}^{*}\|^{2}$ and the boundness \eqref{eq14} of $A_{2}$.

Following the definition of $\tau$, we rewrite $B$ as
\begin{equation}  \label{eq20}
\begin{aligned}
B&=2\alpha\sum\limits_{j=1}^{N}w_{ij}\big[(2L\alpha-1)(f_{j}(\textbf{x}_{j}^{t})-f_{j}^{*})+(f_{j}(\textbf{x}^{*})-f_{j}^{*})\big]\\
&\leq2\alpha(2L\alpha-1)\underbrace{\sum\limits_{j=1}^{N}w_{ij}(f_{j}(\textbf{x}_{j}^{t})-f_{j}^{*})}\limits_{C}+2\alpha\tau
\end{aligned}
\end{equation}
Next, to bound $C$, we have
\begin{equation}  \label{eq21}
\begin{aligned}
C&=\sum\limits_{j=1}^{N}w_{ij}(f_{j}(\textbf{x}_{j}^{t})-f_{j}^{*})\\
&=\sum\limits_{j=1}^{N}w_{ij}\big[\big(f_{j}(\textbf{x}_{j}^{t})-f_{j}(\overline{\textbf{x}}_{j}^{t})\big)+\big(f_{j}(\overline{\textbf{x}}_{j}^{t})-f_{j}^{*}\big)\big]\\
&\geq\sum\limits_{j=1}^{N}w_{ij}\big[\langle\nabla f_{j}(\overline{\textbf{x}}_{j}^{t}),\textbf{x}_{j}^{t}-\overline{\textbf{x}}_{j}^{t}\rangle+\big(f_{j}(\overline{\textbf{x}}_{j}^{t})-f_{j}^{*}\big)\big]\\
&\geq-\frac{1}{2}\sum\limits_{j=1}^{N}w_{ij}\big[\pi\|\nabla f_{j}(\overline{\textbf{x}}_{j}^{t})\|^{2}+\frac{1}{\pi}\|\textbf{x}_{j}^{t}-\overline{\textbf{x}}_{j}^{t}\|^{2}\big]\\
&\;\;+\sum\limits_{j=1}^{N}w_{ij}\big(f_{j}(\overline{\textbf{x}}_{j}^{t})-f_{j}^{*}\big)\\
&\geq-\sum\limits_{j=1}^{N}w_{ij}\big[L\pi(f_{j}(\overline{\textbf{x}}_{j}^{t})-f_{j}^{*})+\frac{1}{2\pi}\|\textbf{x}_{j}^{t}-\overline{\textbf{x}}_{j}^{t}\|^{2}\big]\\
&\;\;+\sum\limits_{j=1}^{N}w_{ij}\big(f_{j}(\overline{\textbf{x}}_{j}^{t})-f_{j}^{*}\big)\\
&\geq-\sum\limits_{j=1}^{N}w_{ij}\big[(L\pi-1)(f_{j}(\overline{\textbf{x}}_{j}^{t})-f_{j}^{*})+\frac{1}{2\pi}\|\textbf{x}_{j}^{t}-\overline{\textbf{x}}_{j}^{t}\|^{2}\big]
\end{aligned}
\end{equation}
where the first inequality is based on the convexity of $f_{j}$, the second inequality follows from the fact of $2\langle\textbf{a},\textbf{b}\rangle\geq-\pi\|\textbf{a}\|^{2}+\pi^{-1}\|\textbf{b}\|^{2}$ for any vectors $\textbf{a}$ and $\textbf{b}$, and $\pi>0$. In the third inequality, we use the L-smooth assumption \ref{assum2} of $f_{j}$. If the condition of $L\pi-1<0$ (i.e., $\pi<\frac{1}{L}$) is satisfied, and by the fact of $f_{j}(\overline{\textbf{x}}_{j}^{t})-f_{j}^{*}\geq0$, the quantity $C$ can be further bounded by
\begin{equation}  \label{eq22}
\begin{aligned}
C\geq-\frac{1}{2\pi}\sum\limits_{j=1}^{N}w_{ij}\|\textbf{x}_{j}^{t}-\overline{\textbf{x}}_{j}^{t}\|^{2}.
\end{aligned}
\end{equation}
Due to $\alpha\leq\frac{1}{2L}$, substituting \eqref{eq22} into \eqref{eq20}, we have
\begin{equation}  \label{eq23}
\begin{aligned}
B&\leq\gamma\sum\limits_{j=1}^{N}w_{ij}\|\textbf{x}_{j}^{t}-\overline{\textbf{x}}_{j}^{t}\|^{2}+2\alpha\tau,
\end{aligned}
\end{equation}
which leads to the result \eqref{eq13} by substituting \eqref{eq19} and \eqref{eq23} into \eqref{eq14}.

\end{proof}

Next, we bound the second and third terms of right hand of inequality \eqref{eq13}.

\begin{lemma}\label{lem2}\textbf{(Bounded deviation $\|\textbf{x}_{j}^{t}-\overline{\textbf{x}}_{j}^{t}\|$)}
Under Assumption \ref{assum3}, for the inexact MUSIC \eqref{eq8}-\eqref{eq9}, $\|\textbf{x}_{j}^{t}-\overline{\textbf{x}}_{j}^{t}\|$ is bounded as
\begin{equation}  \label{eq24}
\begin{aligned}
\|\textbf{x}_{j}^{t}-\overline{\textbf{x}}_{j}^{t}\|\leq2\alpha (t-t^{0})G_{max},
\end{aligned}
\end{equation}
where $t^{0}\in\mathcal{I}_{E}$ means one combination time and satisfies $0\leq t-t^{0}\leq E-1$ for any $t$.
\end{lemma}
\begin{proof}

Firstly, in the case of $t=t^{0}$, the bound \eqref{eq24} always true due to $\textbf{x}_{j}^{t^{0}}=\overline{\textbf{x}}_{j}^{t^{0}}$ based on the combination policy \eqref{eq9}. Secondly, we can write $\|\textbf{x}_{j}^{t}-\overline{\textbf{x}}_{j}^{t}\|$ as
\begin{equation}  \label{eq25}
\begin{aligned}
\|\textbf{x}_{j}^{t}-\overline{\textbf{x}}_{j}^{t}\|&=\|\textbf{x}_{j}^{t}-\textbf{x}_{j}^{t^{0}}+\textbf{x}_{j}^{t^{0}}-\overline{\textbf{x}}_{j}^{t}\|\\
&\leq\|\textbf{x}_{j}^{t}-\textbf{x}_{j}^{t^{0}}\|+\|\overline{\textbf{x}}_{j}^{t}-\textbf{x}_{j}^{t^{0}}\|.
\end{aligned}
\end{equation}
For the inner loop iterations from $t^{0}$ to $t\leq t^{0}+E-1$, we have

\begin{equation}  \label{eq26}
\begin{aligned}
&\textbf{x}_{j}^{t^{0}+1}=\textbf{x}_{j}^{t^{0}}-\alpha\nabla f_{j}(\textbf{x}_{j}^{t^{0}}),\\
&\textbf{x}_{j}^{t^{0}+2}=\textbf{x}_{j}^{t^{0}+1}-\alpha\nabla f_{j}(\textbf{x}_{j}^{t^{0}+1}),\\
&\;\;\;\;\;\;\;\;\;\;\vdots\\
&\textbf{x}_{j}^{t}=\textbf{x}_{j}^{t-1}-\alpha\nabla f_{j}(\textbf{x}_{j}^{t-1}).
\end{aligned}
\end{equation}
Summing over \eqref{eq26} gives
\begin{equation}  \label{eq27}
\begin{aligned}
\textbf{x}_{j}^{t}-\textbf{x}_{j}^{t^{0}}=-\alpha\sum\limits_{s=t^{0}}^{t}\nabla f_{j}(\textbf{x}_{j}^{s}).
\end{aligned}
\end{equation}
Based on Assumption \ref{assum3}, hence we have
\begin{equation}  \label{eq28}
\begin{aligned}
\big\|\textbf{x}_{j}^{t}-\textbf{x}_{j}^{t^{0}}\big\|=\bigg\|\alpha\sum\limits_{s=t^{0}}^{t}\nabla f_{j}(\textbf{x}_{j}^{s})\bigg\|&\leq\alpha(t-t^{0})G_{max}.
\end{aligned}
\end{equation}

By making weighted summation on \eqref{eq26}, it follows that
\begin{equation}  \label{eq29}
\begin{aligned}
&\overline{\textbf{x}}_{j}^{t^{0}+1}=\overline{\textbf{x}}_{j}^{t^{0}}-\alpha\sum\limits_{l=1}^{N}w_{jl}\nabla f_{l}(\textbf{x}_{l}^{t^{0}}),\\
&\overline{\textbf{x}}_{j}^{t^{0}+2}=\overline{\textbf{x}}_{j}^{t^{0}+1}-\alpha\sum\limits_{l=1}^{N}w_{jl}\nabla f_{l}(\textbf{x}_{l}^{t^{0}+1}),\\
&\;\;\;\;\;\;\;\;\;\;\vdots\\
&\overline{\textbf{x}}_{j}^{t}=\overline{\textbf{x}}_{j}^{t-1}-\alpha\sum\limits_{l=1}^{N}w_{jl}\nabla f_{l}(\textbf{x}_{l}^{t-1}).
\end{aligned}
\end{equation}
In the same summing and 2-norm way, we obtain similarly the upper bound
\begin{equation}  \label{eq30}
\begin{aligned}
\big\|\overline{\textbf{x}}_{j}^{t}-\overline{\textbf{x}}_{j}^{t^{0}}\big\|&=\bigg\|\alpha\sum\limits_{s=t^{0}}^{t}\sum\limits_{l=1}^{N}w_{jl}\nabla f_{l}(\textbf{x}_{l}^{s})\bigg\|\\
&\leq\alpha\sum\limits_{s=t^{0}}^{t}\sum\limits_{l=1}^{N}w_{jl}\big\|\nabla f_{l}(\textbf{x}_{l}^{s})\big\|\leq\alpha (t-t^{0})G_{max}.
\end{aligned}
\end{equation}
Due to $\overline{\textbf{x}}_{j}^{t^{0}}=\textbf{x}_{j}^{t^{0}}$, we rewrite \eqref{eq30} as
\begin{equation}  \label{eq31}
\begin{aligned}
\big\|\overline{\textbf{x}}_{j}^{t}-\textbf{x}_{j}^{t^{0}}\big\|\leq\alpha (t-t^{0})G_{max}.
\end{aligned}
\end{equation}
Substituting  \eqref{eq31} and \eqref{eq28} into \eqref{eq25} completes the proof.

\end{proof}

Before bounding $\|\textbf{x}_{j}^{t}-\overline{\textbf{x}}_{i}^{t}\|$, we introduce an additional assumption of bounded disagreement.
\begin{assumption}\label{assum4}
For any iteration $t^{0}\in\mathcal{I}_{E}$ in the inexact MUSIC \eqref{eq8}-\eqref{eq9} and the subsequent exact MUSIC \eqref{eq49}-\eqref{eq50}, the deviations between any two agents $i$ and $j$ are bounded, i.e., $\|\textbf{x}_{j}^{t^{0}}-\textbf{x}_{i}^{t^{0}}\|\leq\varepsilon$, where $\varepsilon$ is a small nonnegative constant.
\end{assumption}
Many previous studies \cite{lobel2010distributed,xi2016distributed,nedic2020distributed,sayed2014adaptive} have clearly shown that  the disagreement between estimates across all agents generated by combination (consensus) step \eqref{eq3} or \eqref{eq6} goes almost surely to zero, i.e., $\lim_{t\rightarrow \infty}\|\textbf{x}_{j}^{t}-\textbf{x}_{i}^{t}\|=0$ with probability 1 when the network connectivity, doubly stochastic weight $w_{ij}$, and bounded gradients assumptions hold. Therefore, the  finite $\|\textbf{x}_{j}^{t^{0}}-\textbf{x}_{i}^{t^{0}}\|$ is a reasonable assumption. Consequently, we get the following lemma.

\begin{lemma}\label{lem3}\textbf{(Bounded disagreement $\|\textbf{x}_{j}^{t}-\overline{\textbf{x}}_{i}^{t}\|$)}
Under Assumption \ref{assum4}, for the inexact MUSIC \eqref{eq8}-\eqref{eq9}, $\|\textbf{x}_{j}^{t}-\overline{\textbf{x}}_{i}^{t}\|$ is bounded as follows
\begin{equation}  \label{eq32}
\begin{aligned}
\|\textbf{x}_{j}^{t}-\overline{\textbf{x}}_{i}^{t}\|\leq4\alpha (t-t^{0})G_{max}+\varepsilon,
\end{aligned}
\end{equation}
for any $t^{0}\in\mathcal{I}_{E}$ and  $0\leq t-t^{0}\leq E-1$.
\end{lemma}

\begin{proof}

Note that
\begin{equation}  \label{eq33}
\begin{aligned}
\|\textbf{x}_{j}^{t}-\overline{\textbf{x}}_{i}^{t}\|&=\|\textbf{x}_{j}^{t}-\overline{\textbf{x}}_{j}^{t}+\overline{\textbf{x}}_{j}^{t}-\overline{\textbf{x}}_{i}^{t}\|\\
&\leq\|\textbf{x}_{j}^{t}-\overline{\textbf{x}}_{j}^{t}\|+\|\overline{\textbf{x}}_{j}^{t}-\overline{\textbf{x}}_{i}^{t}\|
\end{aligned}
\end{equation}
for any two agents $i$ and $j$ over the network.
For the second term of right hand of \eqref{eq33}, we have
\begin{equation}  \label{eq34}
\begin{aligned}
\|\overline{\textbf{x}}_{j}^{t}-\overline{\textbf{x}}_{i}^{t}\|&=\|(\overline{\textbf{x}}_{j}^{t}-\textbf{x}_{j}^{t^{0}})+(\textbf{x}_{i}^{t^{0}}-\overline{\textbf{x}}_{i}^{t})+(\textbf{x}_{j}^{t^{0}}-\textbf{x}_{i}^{t^{0}})\|\\
&\leq\|\overline{\textbf{x}}_{j}^{t}-\textbf{x}_{j}^{t^{0}}\|+\|\textbf{x}_{i}^{t^{0}}-\overline{\textbf{x}}_{i}^{t}\|+\|\textbf{x}_{j}^{t^{0}}-\textbf{x}_{i}^{t^{0}}\|\\
&\leq2\alpha (t-t^{0})G_{max}+\varepsilon,
\end{aligned}
\end{equation}
where we use the inequality \eqref{eq31} and Assumption \ref{assum4}.

Substituting \eqref{eq24} and \eqref{eq34} into \eqref{eq33} leads to \eqref{eq32}.

\end{proof}

Finally, we obtain the convergence result of inexact MUSIC as follows:
\begin{theorem}\label{them1} Let Assumptions \ref{assum1}-\ref{assum4} and $\alpha\leq\frac{1}{2L}$ hold, the inexact MUSIC \eqref{eq8}-\eqref{eq9} converges linearly  in the mean-square sense to a neighborhood of the optimum solution:
\begin{equation}  \label{eq35}
\begin{aligned}
\big\|\overline{\textbf{x}}_{i}^{kE}-\textbf{x}^{*}\big\|^{2}\leq(1-\mu\alpha)^{kE}\big\|\overline{\textbf{x}}_{i}^{0}-\textbf{x}^{*}\big\|^{2}+D_{1}
\end{aligned}
\end{equation}
for $k=1,2,\ldots,\lfloor T/E\rfloor$, where
\begin{equation}  \label{eq35-1}
\begin{aligned}
D_{1}&=\frac{(1-(1-\mu\alpha)^{kE})}{1-(1-\mu\alpha)^{E}}\sum\limits_{s=0}^{E-1}\xi^{E-1-s}(1-\mu\alpha)^{s}\\
&\underrightarrow{k\rightarrow\infty}\;\mathcal{O}\bigg(\frac{(E-1)^{2}(16+4\gamma)\alpha^{2}G_{max}^{2}+2\alpha\tau}{\mu\alpha}\bigg)
\end{aligned}
\end{equation}
and $\xi^{s}=4\gamma\alpha^{2}s^{2}G^{2}_{max}+[4\alpha sG_{max}+\varepsilon]^{2}+2\alpha\tau$.
\end{theorem}

\begin{proof}
It is known that no matter whether $t\in\mathcal{I}_{E}$ or $t\notin\mathcal{I}_{E}$, $\overline{\textbf{x}}_{i}^{t}=\overline{\textbf{v}}_{i}^{t}$ is always tenable. Hence, by combining Lemmas \ref{lem1}-\ref{lem3},  we have
\begin{equation}  \label{eq36}
\begin{aligned}
\parallel\overline{\textbf{x}}_{i}^{t+1}-\textbf{x}^{*}\parallel^{2}\leq(1-\mu\alpha)\|\overline{\textbf{x}}_{i}^{t}-\textbf{x}^{*}\|^{2}+\xi^{t-t^{0}}.
\end{aligned}
\end{equation}

For convenience, when $\Delta^{t}=\|\overline{\textbf{x}}_{i}^{t}-\textbf{x}^{*}\|^{2}$ is defined, we also write \eqref{eq36} as
\begin{equation}  \label{eq38}
\begin{aligned}
\Delta^{t+1}\leq (1-\mu\alpha)\Delta^{t}+\xi^{t-t^{0}}.
\end{aligned}
\end{equation}
By recursively applying \eqref{eq38} from $t^{0}+1$ to $t^{0}+E$, we obtain
\begin{equation}  \label{eq39}
\begin{aligned}
\Delta^{t^{0}+E}&\leq(1-\mu\alpha)^{E}\Delta^{t^{0}}+\sum\limits_{s=0}^{E-1}\xi^{E-1-s}(1-\mu\alpha)^{s},
\end{aligned}
\end{equation}
or,
\begin{equation}  \label{eq40}
\begin{aligned}
\big\|\overline{\textbf{x}}_{i}^{kE}-\textbf{x}^{*}\big\|^{2}\leq(1-\mu\alpha)^{E}\big\|\overline{\textbf{x}}_{i}^{(k-1)E}-\textbf{x}^{*}\big\|^{2}+D_{2},
\end{aligned}
\end{equation}
where $D_{2}=\sum\limits_{s=0}^{E-1}\xi^{E-1-s}(1-\mu\alpha)^{s}$ and $k=1,2,\ldots,\lfloor T/E\rfloor$.
By recursively using  \eqref{eq40} for $k$ times, we have
\begin{equation}  \label{eq41}
\begin{aligned}
\big\|\overline{\textbf{x}}_{i}^{kE}-\textbf{x}^{*}\big\|^{2}\leq(1-\mu\alpha)^{kE}\big\|\overline{\textbf{x}}_{i}^{0}-\textbf{x}^{*}\big\|^{2}+D_{1},
\end{aligned}
\end{equation}
where
\begin{equation}  \label{eq42}
\begin{aligned}
D_{1}=\frac{D_{2}(1-(1-\mu\alpha)^{kE})}{1-(1-\mu\alpha)^{E}}.
\end{aligned}
\end{equation}

When $k\rightarrow\infty$, we can conclude that a consensus is asymptotically achieved among the local estimates (i.e., $\lim\limits_{k\rightarrow\infty}\varepsilon=0$) under the standard consensus strategy. Moreover, we additionally use $\xi^{t-t^{0}}=\xi^{E-1}$ in \eqref{eq38} with a purpose to simplify calculations, thus obtain
\begin{equation}  \label{eq42-1}
\begin{aligned}
\limsup\limits_{k\rightarrow\infty} \xi^{E-1}=(E-1)^{2}(16+4\gamma)\alpha^{2}G_{max}^{2}+2\alpha\tau,
\end{aligned}
\end{equation}
\begin{equation}  \label{eq42-2}
\begin{aligned}
\limsup\limits_{k\rightarrow\infty}D_{2}=\frac{(1-(1-\mu\alpha)^{E})}{\mu\alpha}\limsup\limits_{k\rightarrow\infty}\xi^{E-1}
\end{aligned}
\end{equation}
and
\begin{equation}  \label{eq42-3}
\begin{aligned}
\limsup\limits_{k\rightarrow\infty}D_{1}=\frac{(E-1)^{2}(16+4\gamma)\alpha G_{max}^{2}+2\tau}{\mu},
\end{aligned}
\end{equation}
which completes the proof.
\end{proof}

Provided that $\alpha\leq \frac{1}{2L}$ due to $L\geq \mu$ based on Assumptions \ref{assum1} and \ref{assum2}, Theorem \ref{them1} shows that the mean square solution generated by the inexact MUSIC method converges linearly with a rate $\mathcal{O}((1-\mu\alpha)^{kE})$, which is monotone increasing and decreasing with respect to $E$ and $\alpha$, respectively, until reaching the error neighborhood with size
\begin{equation}  \label{eq42-4}
\begin{aligned}
\mathcal{O}\bigg(\underbrace{\frac{(E-1)^{2} (16+4\gamma)\alpha G_{max}^{2}}{\mu}}\limits_{\textrm{local\;drift}}+\underbrace{\frac{2\tau}{\mu}}\limits_{\textrm{inexact\;bias}}\bigg).
\end{aligned}
\end{equation}
which consists of two terms. The local drift term results from the accumulation of deviations between local variables and the global consensus when examining the second and third terms on the right hand side of \eqref{eq13}. The second term is  the source of the bias generated by the inherent inexact strategy. When $E=1$, the inexact MUSIC degrades to the standard ATC version \eqref{eq5}-\eqref{eq6} with a convergence rate $\mathcal{O}((1-\mu\alpha)^{k})$ and an asymptotic error of size $\mathcal{O}(\frac{2\tau}{\mu})$, which can not be removed in the context of such inexact policy.

\begin{remark} \textbf{(Choices of $\alpha$ and $E$)}On one hand, from Theorem \ref{them1}, there are no restrictions imposed on the frequency $E$ of local updates. This implies that $E$ can take on a large value to expedite convergence. On the other hand, as indicated by \eqref{eq42-4}, an excessively large value of $E$ can significantly expand the size of the error neighborhood.  Consequently, the parameter $E$ plays a role similar to the step size $\alpha$ in balancing the tradeoff between convergence speed and accuracy. As a result, by selecting a slightly larger $E$ than 1 (e.g., 2, 3, 4) along with a small step size, we can achieve a double win of convergence rate and steady-state accuracy. This situation effectively addresses a longstanding challenge in the domain of conventional optimization techniques based on inexact first-order methods. Particularly, it is a better choice by using a diminishing step size (e.g., $\alpha^{t}=\frac{\alpha}{t^{\delta}}, \delta\in(0,2)$) to reinforce this strategy.

\end{remark}

\subsection{Numerical Results for inexact MUSIC}\label{IIB}

In this section, we provide some empirical results of inexact MUSIC for solving a representative least squares problem with the following form
\begin{equation}  \label{eq43}
\min\limits_{\textbf{x}\in\mathbb{R}^{p}}\sum\limits_{i=1}^{N}f_{i}(\textbf{x})=\min\limits_{\textbf{x}\in\mathbb{R}^{p}}\sum\limits_{i=1}^{N}\frac{1}{2}\|\textbf{A}_{i}^{T}\textbf{x}-b_{i}\|^{2}+\frac{\mu}{2}\|\textbf{x}\|^{2},
\end{equation}
where we assume that each agent $i$ holds the local objective $f_{i}(\textbf{x})=\frac{1}{2}\|\textbf{A}_{i}^{T}\textbf{x}-b_{i}\|^{2}+\frac{\mu}{2}\|\textbf{x}\|^{2}$. We generate $\textbf{A}_{i}\in\mathbb{R}^{p\times m}$ and $b_{i}\in\mathbb{R}^{m}$ by following the uniform distribution with each entry in [0, 1]. Based on the global cost function given in \eqref{eq43}, the optimal solution can be obtained as $\textbf{x}^{*}=(\sum_{i=1}^{N}\textbf{A}_{i}^{T}\textbf{A}_{i}+\mu \textbf{I})^{-1}\sum_{i=1}^{N}\textbf{A}_{i}^{T}\textbf{b}_{i}$. We evaluate performance in terms of the relative error that is defined as $\frac{1}{N}\sum_{i=1}^{N}\frac{\|\textbf{x}_{i}^{t}-\textbf{x}^{*}\|^{2}}{\|\textbf{x}_{i}^{0}-\textbf{x}^{*}\|^{2}}$ with initial value $\textbf{x}_{i}^{0}=0$. The weight matrix $\textbf{W}$ over an undirected Erdos-Renyi graph with average degree 4 is generated by Metropolis rule \cite{sayed2014adaptation} since no obvious difference exists between the different doubly stochastic rules. We set $N=100$, $p=m=10$ and $\mu=10^{-6}$ for all experiments in this problem.

\textbf{Effect of $E$.}  From Fig. \ref{fig2} (a), when the step size is fixed during iterations, one can see that the parameter $E$ plays a role similar to the step size $\alpha$ (see Fig. \ref{fig2} (b)), i.e., larger (smaller) $E$ or $\alpha$ results in faster (slower) convergence rate and lower (higher) accuracy. Unlike conventional ATC/DGD method, where only step size parameter is used to control the convergence of algorithm, our inexact MUSIC provides a new tool enabling balance between rate and accuracy for inexact methods, such as both fast rate and good accuracy can be achieved.

\begin{figure}
\centering
\subfigure[ ]{
\centering
\includegraphics[width=4cm,height=3 cm]{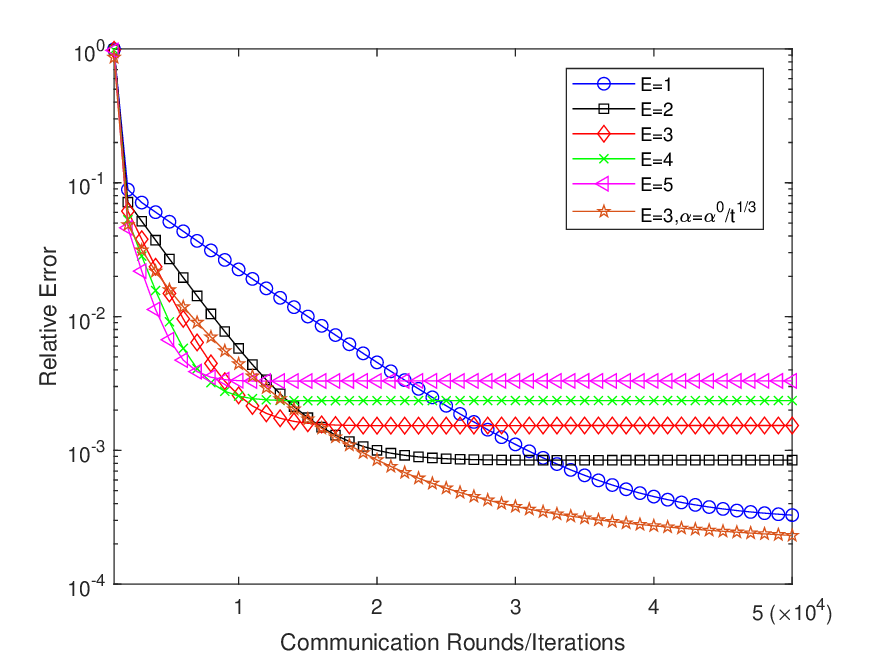}
}
\subfigure[ ]{
\centering
\includegraphics[width=4cm,height=3 cm]{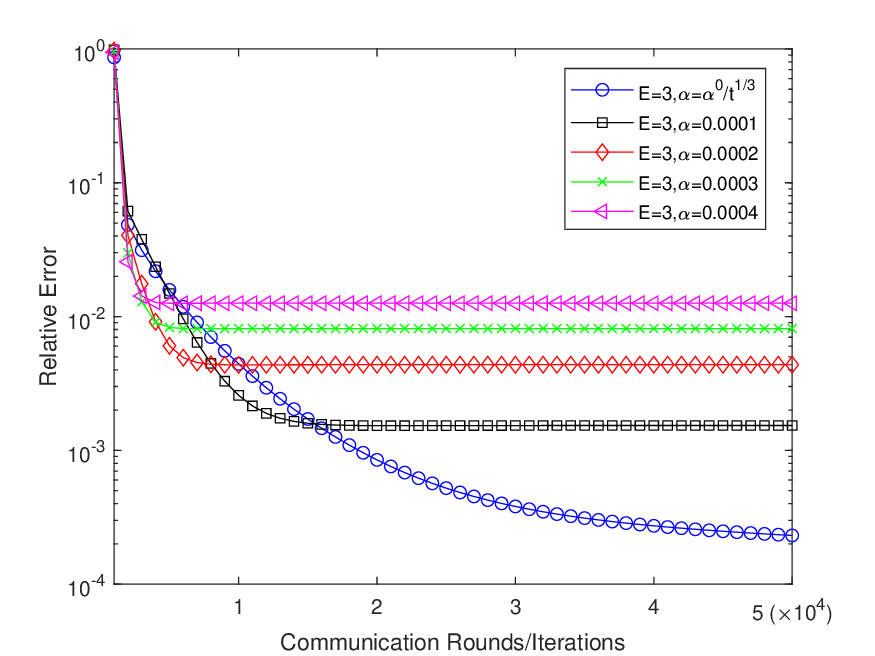}
}
\subfigure[ ]{
\centering
\includegraphics[width=4cm,height=3 cm]{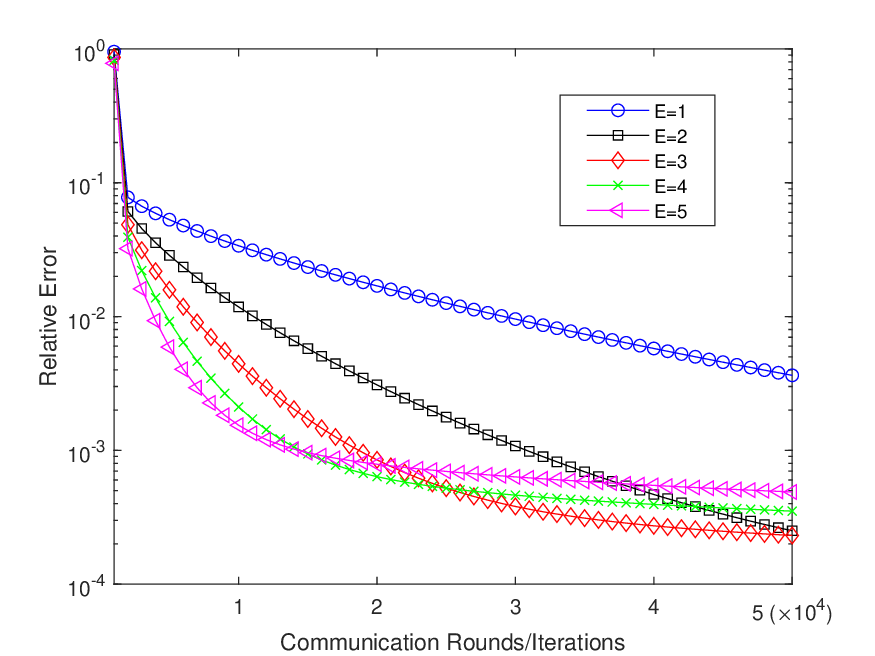}
}\centering
\subfigure[ ]{
\centering
\includegraphics[width=4cm,height=3 cm]{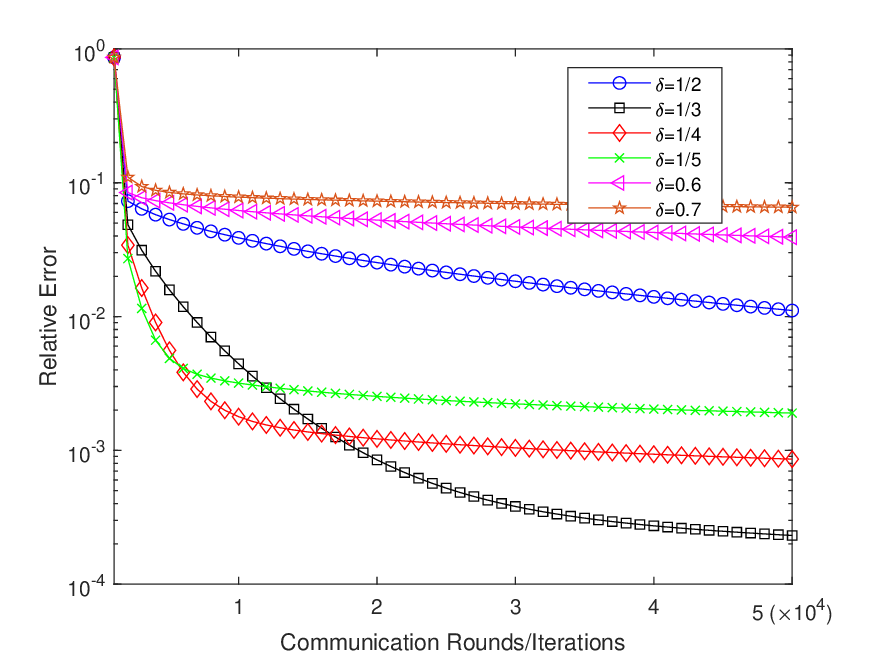}
}\centering
\caption{ Performance of inexact MUSIC measured in terms of relative error with respect to communication rounds or iterations on the distributed quadratic problem (a) impact of $E$ under $\alpha=0.0001$ (b) impact of different fixed step sizes under $E=3$ (c) impact of $E$ under a diminishing step size $\alpha=\alpha^{0}/t^{\frac{1}{2}}$ with $\alpha^{0}=0.001$ (d) impact of diminishing step sizes under $E=3$.}\label{fig2}
\end{figure}

\textbf{Benefits of diminishing step sizes.} Fig. \ref{fig2} (a) and (b) show that the diminishing step size achieves the best performance both on rate and accuracy. Under a diminishing step size,  Fig. \ref{fig2} (c) shows that over-large or over-small $E$ leads to significant worse convergence accuracy. When $E$ is fixed, same effect on the decaying rate $\delta$ is also observed in Fig. \ref{fig2} (d), which consolidates  the efficiency of $E$.

\section{Exact MUSIC}\label{IV}

Though serving as a warm up method, the feasibility of inexact MUSIC motivates us to ask the question whether exact convergence with communication efficacy can be achieved in a MUSIC way.  Obviously, previous results from inexact MUSIC  indicate that only multiple updates are insufficient for converging to the exact solution. Instead, a larger $E$ leads to a larger error neighborhood. Several recent works on exact methods have been proposed, such as EXTRA, DIGing, NEAR$\_$DGD, etc. However, their exact solutions are achieved at the cost of expensive communication. Table \ref{tab1} shows a comparison on the number of communications (gradient exchange or decision vector exchange) per round to reach an exact solution.

\begin{table*}
  \centering
  \caption{A comparison of existing  representative distributed algorithms when they converge to an exact solution in terms of communications and gradient evaluations. Here, $p\;\textrm{or}\;2p$ represents that $p$ or $2p$ scalar communications are consumed when extra memory is used or not. The same explanation is given to the notation of $1\;\textrm{or}\;2$. $\kappa\triangleq\frac{L}{\mu}>1$ is the condition number of the objective function and $0<\rho<1$ is the spectral radius of  the network.}\label{tab1}
  \begin{tabular}{ccccc}
  \hline
  Algorithm  & Communicated scalars per agent & Numbers of gradient evaluations & Communication complexity\\
                            &during one round &    per agent during one round &  \\
     \hline
 Exact MUSIC (this paper) & $p$  & $E$ & $\mathcal{O}(\frac{2\kappa}{E}\log (\frac{1}{\epsilon}))$ \\
Algorithm (\ref{eq86})(\ref{eq87})& $p$  & $E$ & $\mathcal{O}(\frac{2\kappa}{E}\log (\frac{1}{\epsilon}))$\\
Exact diffusion \cite{yuan2018exact,yuan2020influence}  & $p$ &  $1$ & $\mathcal{O}(2\kappa\log (\frac{1}{\epsilon}))$ \\
  EXTRA \cite{shi2015extra,li2020revisiting} & $p\;\textrm{or}\;2p$ &  $1\;\textrm{or}\;2$ & $\mathcal{O}(\frac{L^{2}\kappa^{2}}{1-\rho}\log (\frac{1}{\epsilon}))$ \\
  DIGing \cite{nedic2017achieving} & $2p$  &  $1\;\textrm{or}\;2$ & $\mathcal{O}(\frac{\kappa}{(1-\rho)^{2}}\log (\frac{1}{\epsilon}))$\\
Aug-DGM  \cite{7963560,7402509} & $2p$ & $1\;\textrm{or}\;2$ & $\mathcal{O}(\max\{\kappa,\frac{1}{(1-\rho)^{2}}\}\log (\frac{1}{\epsilon}))$\\
NIDS \cite{li2019decentralized}& $p\;\textrm{or}\;2p$ &  $1\;\textrm{or}\;2$   & $\mathcal{O}(\max\{\kappa,\frac{1}{1-\rho}\}\log (\frac{1}{\epsilon}))$\\
Harnessing \cite{qu2017harnessing} & $2p$ &  $1\;\textrm{or}\;2$   & $\mathcal{O}(\frac{\kappa}{(1-\rho)^{2}}\log (\frac{1}{\epsilon}))$\\
NEAR-DGD$^{+}$ \cite{berahas2018balancing} & $cp\;\;(c\gg1)$ &  $c\gg1$ & $\mathcal{O}((\log (\frac{1}{\epsilon}))^{2})$ \\
Gradient tracking \cite{nedic2017achieving,alghunaim2020decentralized} & $2p$ & $1\;\textrm{or}\;2$ & $\mathcal{O}((\kappa+\frac{1}{(1-\rho)^{2}})\log (\frac{1}{\epsilon}))$\\
\hline
\end{tabular}
\end{table*}

For the purpose of communication efficiency, we aim to develop a novel exact MUSIC method based on the excellent exact diffusion scheme. Without any increase in communication as the inexact MUSIC method, the proposed method is communication efficient and exactly converges to the optimal solution. The main challenge is to ensure that nodes can still converge or approach to the optimal solution, while multiple local iterations are performed. Originating from the ATC structure, the vanilla exact diffusion method embeds a correction step between the local update and combination steps, as depicted below:

\begin{align}
\label{eq44} &\textbf{v}^{t+1}_{i}=\textbf{x}^{t}_{i}-\alpha \nabla f_{i}(\textbf{x}^{t}_{i}),\;\;\;\;\;\textbf{(local update)}\\
\label{eq45} &\textbf{y}^{t+1}_{i}=\textbf{v}^{t+1}_{i}+\textbf{x}^{t}_{i}-\textbf{v}^{t}_{i},\;\;\;\;\;\textbf{(correct)}\\
\label{eq46} &\textbf{x}^{t+1}_{i}=\sum\limits_{j\in\mathcal{N}_{i}}w_{ij}\textbf{y}^{t+1}_{j}. \;\;\;\;\;\textbf{(combine)}
\end{align}

In this adapt-correct-combine (ACC) structure, the correction means that the difference between local update and global combination at previous iteration is removed, such that the local estimate  is closer to the global one. Meanwhile, compared with the ATC method \eqref{eq5}-\eqref{eq6}, the exact diffusion \eqref{eq5}-\eqref{eq6} has the same number of communications and gradient evaluations, and  slightly more computation. To be precise, $2p$ additional additions per agent at each iteration are performed in the correction step.

Neglecting the intermediate variable $\textbf{y}^{t+1}_{i}$, we also can blend the exact diffusion in two steps
\begin{align}
\label{eq47} &\textbf{v}^{t+1}_{i}=\textbf{x}^{t}_{i}-\alpha \nabla f_{i}(\textbf{x}^{t}_{i}),\;\;\;\;\;\textbf{(local update)}\\
\label{eq48} &\textbf{x}^{t+1}_{i}=\sum\limits_{j\in\mathcal{N}_{i}}\overline{w}_{ij}(\textbf{v}^{t+1}_{j}+\textbf{x}^{t}_{j}-\textbf{v}^{t}_{j}) \;\;\;\;\;\textbf{(combine)}.
\end{align}

Based on the strategy of multi-updates single-combination, in this section, we aim to achieve a faster convergence rate while maintaining exact implementation and high communication efficiency at the cost of more local computations. Our proposed exact MUSIC algorithm updates as follows:

\begin{align}
\label{eq49} &\textbf{v}^{t+1}_{i}=\textbf{x}^{t}_{i}-\alpha \nabla f_{i}(\textbf{x}^{t}_{i}),\\
\label{eq50} &\textbf{x}_{i}^{t+1}=\begin{cases}
\textbf{v}^{t+1}_{i}+\beta(\textbf{x}^{t^{0}}_{i}-\textbf{v}^{t^{0}}_{i}) &\textrm{if}\;\; t+1\notin \mathcal{I}_{E} \\
\sum\limits_{j\in \mathcal{N}_{i}}\overline{w}_{ij}(\textbf{v}^{t+1}_{j}+\beta(\textbf{x}^{t^{0}}_{j}-\textbf{v}^{t^{0}}_{j}))  &\textrm{if}\;\; t+1\in \mathcal{I}_{E}
\end{cases},
\end{align}
where $t^{0}$ has a same definition as in the inexact MUSIC or $t^{0}=t+1-(t+1)\%E$ in this context, $\beta\in[0,1]$ is a gain factor which control the ratio of bias compensation to avoid overcompensation as the increase of $E$.

Fig. \ref{fig3} illustrates information flow in our exact MUSIC that  also contains two types of iterations. Compared to the inexact MUSIC, the main difference is that each agent corrects the bias between the local estimate and the previous global combination using the rule given in \eqref{eq45}. In other words, exact MUSIC uses explicitly multi-step corrections matching multi-step gradient descent at each agent. Moreover, in exact MUSIC,  the combination matrix $\overline{\textbf{W}}=(\textbf{W}+\textbf{I}_{N})/2=[\overline{w}_{ij}]\in \mathbb{R}^{N\times N}$ is different from $\textbf{W}$ used in inexact MUSIC. Hence, both $\overline{\textbf{W}}$ and $\textbf{W}$ are symmetric and doubly stochastic.

\begin{figure}[htbp]
\centering
\centerline{\includegraphics[width=5.5cm,height=7.5cm]{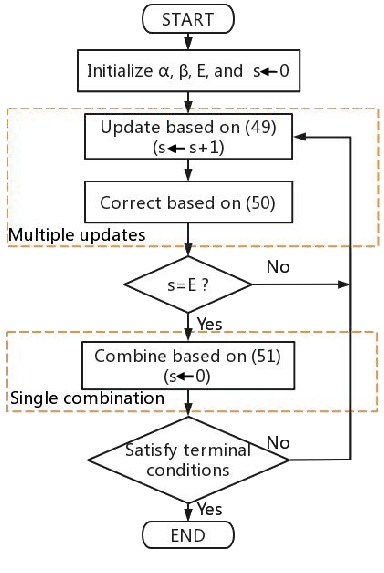}}
\caption{Illustration of workflow in the exact MUSIC.}\label{fig3}
\end{figure}

\subsection{Convergence analysis of exact MUSIC}\label{IIIA}

In this section, except for again using the assumptions \ref{assum1}-\ref{assum4} and the definitions of $\overline{\textbf{v}}_{i}^{t}$, $\overline{\textbf{x}}_{i}^{t}$ and $\overline{\textbf{g}}_{i}^{t}$ for analysis, we introduce several new global variables which are defined as follows:

$$\textbf{x}^{t}=[\textbf{x}_{1}^{t},\textbf{x}_{2}^{t},\ldots,\textbf{x}_{N}^{t}]^{T}\in\mathbb{R}^{Np},$$
$$\textbf{v}^{t}=[\textbf{v}_{1}^{t},\textbf{v}_{2}^{t},\ldots,\textbf{v}_{N}^{t}]^{T}\in\mathbb{R}^{Np},$$
$$\textbf{v}^{*}=[\textbf{v}_{1}^{*},\textbf{v}_{2}^{*},\ldots,\textbf{v}_{N}^{*}]^{T}\in\mathbb{R}^{Np},$$
$$\textbf{g}^{t}=[\textbf{g}_{1}^{t},\textbf{g}_{2}^{t},\ldots,\textbf{g}_{N}^{t}]^{T}\in\mathbb{R}^{Np},$$
where $\textbf{v}_{i}^{*}$ is the optimal vector for minimizing the local objective $f_{i}$ and definitely exists due to the convexity of $f_{i}$. By defining a matrix $\textbf{Z}=\overline{\textbf{W}}\otimes \textbf{I}_{p}\in\mathbb{R}^{Np\times Np}$, it is known that the eigenvalues of $\textbf{Z}$ are the same as those of $\overline{\textbf{W}}$ belonging to $(-1,1]$ due to the fact that the eigenvalues of arbitrary doubly stochastic matrix are bounded in $(-1,1]$.

Further, we can write the update of exact MUSIC \eqref{eq49}-\eqref{eq50} from a global perspective as follows
\begin{equation}  \label{eq51}
\textbf{v}^{t+1}=\textbf{x}^{t}-\alpha \textbf{g}^{t},\;\;\;\;\;\;\;\;\;\;\;\;\;\;\;\;\;\;\;\;\;\;\;\;\;\;\;\;\;\;\;\;\;\;\;\;\;\;\;\;\;\;\;\;\;\;\;\;\;\;
\end{equation}
\begin{equation}  \label{eq52}
\begin{aligned}
\textbf{x}^{t+1}=\begin{cases}
\textbf{v}^{t+1}+\beta(\textbf{x}^{t^{0}}-\textbf{v}^{t^{0}}) &\textrm{if}\;\; t+1\notin \mathcal{I}_{E} \\
\textbf{Z}(\textbf{v}^{t+1}+\beta(\textbf{x}^{t^{0}}-\textbf{v}^{t^{0}}))  &\textrm{if}\;\; t+1\in \mathcal{I}_{E}
\end{cases}.
\end{aligned}
\end{equation}

 From the correction step \eqref{eq50}, it follows that
\begin{equation}  \label{eq53}
\begin{aligned}
\|\overline{\textbf{x}}_{i}^{t+1}-\textbf{x}^{*}\|&=\|\overline{\textbf{v}}_{i}^{t+1}+\overline{\textbf{x}}_{i}^{t^{0}}-\overline{\textbf{v}}_{i}^{t^{0}}-\textbf{x}^{*}\|\\
&\leq\|\overline{\textbf{v}}_{i}^{t+1}-\textbf{x}^{*}\|+\|\overline{\textbf{x}}_{i}^{t^{0}}-\overline{\textbf{v}}_{i}^{t^{0}}\|.
\end{aligned}
\end{equation}
We also note that Lemma \ref{lem1} holds for both inexact and exact methods with one same gradient descent step. Thus, based on observations for inequalities  \eqref{eq13} and \eqref{eq53}, our analysis depends on the following three key lemmas to bound $\|\overline{\textbf{x}}_{i}^{t^{0}}-\overline{\textbf{v}}_{i}^{t^{0}}\|$, $\|\textbf{x}_{j}^{t}-\overline{\textbf{x}}_{i}^{t}\|$ and $\|\textbf{x}_{j}^{t}-\overline{\textbf{x}}_{j}^{t}\|$, respectively. We first establish an important inequality for $\|\overline{\textbf{x}}_{i}^{t^{0}}-\overline{\textbf{v}}_{i}^{t^{0}}\|$ which is the quantity of bias correction applied after each local update step.

\begin{lemma}\label{lem4}\textbf{(Bounded bias correction $\|\overline{\textbf{x}}_{i}^{t^{0}}-\overline{\textbf{v}}_{i}^{t^{0}}\|$)}
Under Assumptions \ref{assum1} and \ref{assum2}, if the step size satisfies $\alpha\leq\frac{1}{2L}$, for the exact MUSIC \eqref{eq49}-\eqref{eq50}, then we have
\begin{equation}  \label{eq54}
\begin{aligned}
\|\overline{\textbf{x}}_{i}^{t^{0}}-\overline{\textbf{v}}_{i}^{t^{0}}\|\leq\Theta^{t^{0}},
\end{aligned}
\end{equation}
where
\begin{equation}  \label{eq55}
\begin{aligned}
\Theta^{t^{0}}= a(x_{1})^{\frac{t^{0}}{E}}+b(x_{2})^{\frac{t^{0}}{E}}+c,
\end{aligned}
\end{equation}
$a$, $b$ and $c$ are the solution of the following linear system
\begin{equation}  \label{eq62-5}
  \begin{bmatrix}
    \Theta^{0} \\
    \Theta^{E}\\
    \Theta^{2E}
  \end{bmatrix}
=
  \begin{bmatrix}
    1 & 1 & 1\\
    x_{1} & x_{2} & 1\\
    (x_{1})^{2} & (x_{2})^{2} & 1
  \end{bmatrix}
  \begin{bmatrix}
    a\\
    b\\
    c
  \end{bmatrix}
\end{equation}
with
\begin{equation}  \label{eq62-1}
\begin{aligned}
x_{2,1}=\frac{(a_{11}+a_{22})\pm \sqrt{(a_{11}+a_{22})^{2}-4(a_{11}a_{22}-a_{12}a_{21})}}{2},
\end{aligned}
\end{equation}
\begin{equation}  \label{eq61}
\begin{aligned}
  &\begin{bmatrix}
    a_{11} & a_{12} & a_{13} \\
    a_{21} & a_{22} & a_{23}
  \end{bmatrix}\\
&\triangleq
  \begin{bmatrix}
    \frac{\beta\nu(1-\nu^{E})}{1-\nu}\|\textbf{Z}-\textbf{I}\|+\beta\|\textbf{Z}\| & \nu^{E}\|\textbf{Z}-\textbf{I}\| & \|\textbf{Z}-\textbf{I}\|\|\textbf{v}^{*}\| \\
    \frac{\beta\nu(1-\nu^{E})}{1-\nu} & \nu^{E} & 0
  \end{bmatrix}
\end{aligned}
\end{equation}
and
 \begin{equation}  \label{eq62-4}
\begin{aligned}
\begin{cases}
\Theta^{0}&=\|\textbf{x}^{0}-\textbf{v}^{0}\|=0\\
\Theta^{E}&=\nu^{E}\|\textbf{Z}-\textbf{I}\|\|\textbf{v}^{0}-\textbf{v}^{*}\|+\|\textbf{Z}-\textbf{I}\|\|\textbf{v}^{*}\|\\
\Theta^{2E}&=\big[\frac{\nu(1-\nu^{E})}{1-\nu}\|\textbf{Z}-\textbf{I}\|+\|\textbf{Z}\|\big]\beta\Theta^{E}\\
&\;\;\;\;+\nu^{2E}\|\textbf{Z}-\textbf{I}\|\|\textbf{v}^{0}-\textbf{v}^{*}\|+\|\textbf{Z}-\textbf{I}\|\|\textbf{v}^{*}\|
\end{cases},
\end{aligned}
\end{equation}
$\nu=\sqrt{1-2\alpha\lambda}\in(0,1),$ $\lambda=\frac{\mu L}{\mu +L},$ and $\textbf{v}^{*}$ has the entry $\textbf{v}_{i}^{*}=\arg\min_{\textbf{x}_{i}} f_{i}(\textbf{x}_{i})$ for $i=1,2,\ldots,N$.

\end{lemma}

\begin{proof}

We first bound the global difference $\|\textbf{x}^{t^{0}}-\textbf{v}^{t^{0}}\|$ at each combination step
\begin{equation}  \label{eq56}
\begin{aligned}
\|\textbf{x}^{t^{0}}-\textbf{v}^{t^{0}}\|&=\|\textbf{Z}(\textbf{v}^{t^{0}}+\beta(\textbf{x}^{t^{0}-E}-\textbf{v}^{t^{0}-E}))-\textbf{v}^{t^{0}}\|\\
&=\|(\textbf{Z}-\textbf{I})\textbf{v}^{t^{0}}+\beta\textbf{Z}(\textbf{x}^{t^{0}-E}-\textbf{v}^{t^{0}-E})\|\\
&=\|(\textbf{Z}-\textbf{I})(\textbf{v}^{t^{0}}-\textbf{v}^{*})+(\textbf{Z}-\textbf{I})\textbf{v}^{*}\\
&\;\;\;\;\;\;+\beta\textbf{Z}(\textbf{x}^{t^{0}-E}-\textbf{v}^{t^{0}-E})\|\\
&\leq\|\textbf{Z}-\textbf{I}\|\|\textbf{v}^{t^{0}}-\textbf{v}^{*}\|+\|\textbf{Z}-\textbf{I}\|\|\textbf{v}^{*}\|\\
&\;\;\;\;\;\;+\beta\|\textbf{Z}\|\|\textbf{x}^{t^{0}-E}-\textbf{v}^{t^{0}-E}\|.
\end{aligned}
\end{equation}
Here in the first equality, we use the second formula of combination update \eqref{eq52}. Moreover, it is known that  $\|\textbf{Z}\|\leq1$ and $\|\textbf{Z}-\textbf{I}\|\leq2$.

Next we analyze $\|\textbf{v}^{t^{0}}-\textbf{v}^{*}\|$. Based on the global update step \eqref{eq51}, we have
\begin{equation}  \label{eq57}
\begin{aligned}
\|\textbf{v}^{t^{0}}-\textbf{v}^{*}\|&=\|\textbf{x}^{t^{0}-1}-\alpha \textbf{g}^{t^{0}-1}-\textbf{v}^{*}\|\\
&=\sqrt{\sum_{i=1}^{N}\|\textbf{x}_{i}^{t^{0}-1}-\alpha \textbf{g}_{i}^{t^{0}-1}-\textbf{v}_{i}^{*}\|^{2}}\\
&\leq\sqrt{\sum_{i=1}^{N}(1-2\alpha\lambda )\|\textbf{x}_{i}^{t^{0}-1}-\textbf{v}_{i}^{*}\|^{2}}\\
&=\sqrt{1-2\alpha \lambda}\|\textbf{x}^{t^{0}-1}-\textbf{v}^{*}\|\\
&=\sqrt{1-2\alpha \lambda}\|\textbf{v}^{t^{0}-1}+\beta(\textbf{x}^{t^{0}-E}-\textbf{v}^{t^{0}-E})-\textbf{v}^{*}\|\\
&\leq\nu\big(\|\textbf{v}^{t^{0}-1}-\textbf{v}^{*}\|+\beta\|\textbf{x}^{t^{0}-E}-\textbf{v}^{t^{0}-E}\|\big),
\end{aligned}
\end{equation}
where the first inequality follows the standard result for the gradient descent method (Theorem 2.1.15 of \cite{nesterov2003introductory}), i.e., $\|\textbf{x}_{i}^{t^{0}-1}-\alpha \textbf{g}_{i}^{t^{0}-1}-\textbf{v}_{i}^{*}\|\leq\sqrt{1-2\alpha\lambda}\|\textbf{x}_{i}^{t^{0}-1}-\textbf{v}_{i}^{*}\|$ holds under $\alpha\leq\min\{\frac{1}{2L},\frac{2}{\mu+L}\}=\frac{1}{2L}$ and $ \lambda=\frac{\mu L}{\mu +L}$, the first combination update in \eqref{eq52} is used in the fourth equality.

By iteratively applying \eqref{eq57} for $E$ times, it can be obtained that
\begin{equation}  \label{eq58}
\begin{aligned}
\|\textbf{v}^{t^{0}}&-\textbf{v}^{*}\|\leq\nu^{E}\|\textbf{v}^{t^{0}-E}-\textbf{v}^{*}\|+\beta\sum_{s=1}^{E}\nu^{s}\|\textbf{x}^{t^{0}-E}-\textbf{v}^{t^{0}-E}\|\\
&=\nu^{E}\|\textbf{v}^{t^{0}-E}-\textbf{v}^{*}\|+\frac{\beta\nu(1-\nu^{E})}{1-\nu}\|\textbf{x}^{t^{0}-E}-\textbf{v}^{t^{0}-E}\|.
\end{aligned}
\end{equation}

Combining \eqref{eq56} and \eqref{eq58}, we can write
\begin{equation}  \label{eq59}
\begin{aligned}
\begin{cases}
\Phi^{(k+1)E}\leq\big[\frac{\nu(1-\nu^{E})}{1-\nu}\|\textbf{Z}-\textbf{I}\|+\|\textbf{Z}\|\big]\beta\Phi^{kE}\\
\;\;\;\;\;\;\;\;\;\;\;\;\;\;\;\;\;\;+\nu^{E}\|\textbf{Z}-\textbf{I}\|\Psi^{kE}+\|\textbf{Z}-\textbf{I}\|\|\textbf{v}^{*}\|\\
\Psi^{(k+1)E}\leq\frac{\beta\nu(1-\nu^{E})}{1-\nu}\Phi^{kE}+\nu^{E}\Psi^{kE}
\end{cases},
\end{aligned}
\end{equation}
where we define $\Phi^{kE}=\|\textbf{x}^{t^{0}}-\textbf{v}^{t^{0}}\|$ and $\Psi^{kE}=\|\textbf{v}^{t^{0}}-\textbf{v}^{*}\|$ by setting $t^{0}=kE$ with $k=0,1,2,\ldots,t^{0}/E$ in \eqref{eq56} and \eqref{eq58}. We can rewrite \eqref{eq59} as a bilinear recurrence relation with a  generic form as follows
\begin{equation}  \label{eq60}
\begin{aligned}
\begin{cases}
\Phi^{(k+1)E}\leq a_{11}\Phi^{kE}+a_{12}\Psi^{kE}+a_{13}\\
\Psi^{(k+1)E}\leq a_{21} \Phi^{kE}+a_{22}\Psi^{kE}+a_{23}
\end{cases}.
\end{aligned}
\end{equation}

Next, we aim to obtain a general expression of $\Phi^{kE}$. From the proof given in the supplementary document of this work, it is shown that the solution of \eqref{eq60} is determined by the roots of the following formulate
\begin{equation}  \label{eq62}
\begin{aligned}
x^{2}-(a_{11}+a_{22})x+(a_{11}a_{22}-a_{12}a_{21})=0.
\end{aligned}
\end{equation}
By solving \eqref{eq62}, we obtain its two different and nonnegative roots $x_{2}>x_{1}>0$ given in \eqref{eq62-1}.

It is noted that the radical expression of right hand of \eqref{eq62-1} is always valid under $\alpha\leq\frac{1}{2L}$. By giving the general formula of $\Theta^{t^{0}}$ as \eqref{eq55} with initialization \eqref{eq62-4} which is obtained from \eqref{eq59}, then we get the coefficients $a$, $b$, $c$ by solving the equations \eqref{eq62-5}.

Because of $\|\overline{\textbf{x}}_{i}^{t^{0}}-\overline{\textbf{v}}_{i}^{t^{0}}\|\leq\|\textbf{Z}\textbf{x}^{t^{0}}-\textbf{Z}\textbf{v}^{t^{0}}\|\leq\|\textbf{Z}\|\|\textbf{x}^{t^{0}}-\textbf{v}^{t^{0}}\|\leq\Theta^{t^{0}}$, we complete the proof. Moreover, from \eqref{eq59}, the individual disagreement is also bounded
\begin{equation}  \label{eq63}
\begin{aligned}
\|\textbf{x}_{i}^{t^{0}}-\textbf{v}_{i}^{t^{0}}\|\leq\|\textbf{x}^{t^{0}}-\textbf{v}^{t^{0}}\|\leq \Theta^{t^{0}},
\end{aligned}
\end{equation}
which will be used in subsequent analysis.
\end{proof}

The following corollary demonstrates that the bias correction is upper bounded by a constant.
\begin{corollary}\label{corollary1} Given by the relation \eqref{eq54} and the definition \eqref{eq55} in Lemma \ref{lem4}, for the exact MUSIC \eqref{eq49}-\eqref{eq50} with the number $E$ of local updates satisfying
\begin{equation}  \label{eq63-33}
\begin{aligned}
\nu^{E}\big(1-\frac{\beta\nu}{1-\nu}\|\textbf{Z}-\textbf{I}\|\big)\leq1-\frac{\beta\nu}{1-\nu}\|\textbf{Z}-\textbf{I}\|-\beta\|\textbf{Z}\|,
\end{aligned}
\end{equation}
the bias correction $\|\overline{\textbf{x}}_{i}^{t^{0}}-\overline{\textbf{v}}_{i}^{t^{0}}\|$ is bounded by a positive constant
\begin{equation}  \label{eq63-2}
\begin{aligned}
\|\overline{\textbf{x}}_{i}^{t^{0}}-\overline{\textbf{v}}_{i}^{t^{0}}\|\leq\Theta^{t^{0}}\leq\Theta.
\end{aligned}
\end{equation}

\end{corollary}

\begin{proof}

Please see Section II in the supplementary document of this work.

\end{proof}

Corollary \ref{corollary1} shows that the bias correction at any combination step $t^{0}\in \mathcal{I}_{E}$ is bounded when the number of local updates is finite. Further, Corollary \ref{corollary1} also implies that there exists some $\theta>0$ such that if $\|\overline{\textbf{x}}_{i}^{t^{0}}-\overline{\textbf{v}}_{i}^{t^{0}}\|>0$, then
\begin{equation}  \label{eq63-1}
\begin{aligned}
\theta\leq\|\overline{\textbf{x}}_{i}^{t^{0}}-\overline{\textbf{v}}_{i}^{t^{0}}\|\leq\Theta.
\end{aligned}
\end{equation}

Next, we bound $\|\textbf{x}_{j}^{t}-\overline{\textbf{x}}_{j}^{t}\|$ by the following Lemma.

\begin{lemma}\label{lem5}\textbf{(Bounded deviation $\|\textbf{x}_{j}^{t}-\overline{\textbf{x}}_{j}^{t}\|$)}
Under Assumption \ref{assum3}, for the exact MUSIC \eqref{eq49}-\eqref{eq50}, it follows that
\begin{equation}  \label{eq64}
\begin{aligned}
\|\textbf{x}_{j}^{t}-\overline{\textbf{x}}_{j}^{t}\|^{2}\leq4(t-t^{0})^{2}\Gamma
\end{aligned}
\end{equation}
with $0\leq t-t^{0}\leq E-1$, where $\Gamma=\alpha^{2}G_{max}^{2}+\Theta^{2}+\frac{\alpha G_{max}^{2}}{\mu}-\frac{\alpha G_{min}^{2}}{L}-\mu\alpha\beta^{2}\theta^{2}$.
\end{lemma}

\begin{proof}

Due to $\textbf{x}_{j}^{t^{0}}=\overline{\textbf{x}}_{j}^{t^{0}}$ for any $t^{0}\in\mathcal{I}_{E}$ (i.e., $t=t^{0}$), based on the combination step in \eqref{eq50}, the inequality \eqref{eq64} always holds. In the case of  $1\leq t-t^{0}\leq E-1$, it holds that
\begin{equation}  \label{eq65}
\begin{aligned}
\|\textbf{x}_{j}^{t}-\overline{\textbf{x}}_{j}^{t}\|\leq\|\textbf{x}_{j}^{t}-\textbf{x}_{j}^{t^{0}}\|+\|\overline{\textbf{x}}_{j}^{t}-\textbf{x}_{j}^{t^{0}}\|.
\end{aligned}
\end{equation}
We first bound the first term of \eqref{eq65}. For the inner loop iterations from $t^{0}$ to $t$, based on the exact update \eqref{eq49}-\eqref{eq50}, we have
\begin{equation}  \label{eq66}
\begin{aligned}
&\textbf{x}_{j}^{t^{0}+1}=\textbf{x}_{j}^{t^{0}}-\alpha\nabla f_{j}(\textbf{x}_{j}^{t^{0}})+\beta(\textbf{x}_{j}^{t^{0}}-\textbf{v}_{j}^{t^{0}}),\\
&\textbf{x}_{j}^{t^{0}+2}=\textbf{x}_{j}^{t^{0}+1}-\alpha\nabla f_{j}(\textbf{x}_{j}^{t^{0}+1})+\beta(\textbf{x}_{j}^{t^{0}}-\textbf{v}_{j}^{t^{0}}),\\
&\;\;\;\;\;\;\;\;\;\;\vdots\\
&\textbf{x}_{j}^{t}=\textbf{x}_{j}^{t-1}-\alpha\nabla f_{j}(\textbf{x}_{j}^{t-1})+\beta(\textbf{x}_{j}^{t^{0}}-\textbf{v}_{j}^{t^{0}}).
\end{aligned}
\end{equation}
By summing over \eqref{eq66}, it follows that
\begin{equation}  \label{eq67}
\begin{aligned}
\textbf{x}_{j}^{t}-\textbf{x}_{j}^{t^{0}}=-\alpha\sum\limits_{s=t^{0}}^{t}\nabla f_{j}(\textbf{x}_{j}^{s})+\beta(t-t^{0})(\textbf{x}_{j}^{t^{0}}-\textbf{v}_{j}^{t^{0}}).
\end{aligned}
\end{equation}
Taking  the squared 2-norm on \eqref{eq67}, it follows that
\begin{equation}  \label{eq68}
\begin{aligned}
\big\|\textbf{x}_{j}^{t}-\textbf{x}_{j}^{t^{0}}\big\|^{2}=&\big\|\alpha\sum\limits_{s=t^{0}}^{t}\nabla f_{j}(\textbf{x}_{j}^{s})\big\|^{2}+\big\|\beta(t-t^{0})(\textbf{x}_{j}^{t^{0}}-\textbf{v}_{j}^{t^{0}})\big\|^{2}\\
&-2 \langle\alpha\sum\limits_{s=t^{0}}^{t}\nabla f_{j}(\textbf{x}_{j}^{s}),\beta(t-t^{0})(\textbf{x}_{j}^{t^{0}}-\textbf{v}_{j}^{t^{0}})\rangle\\
=&\big\|\alpha\sum\limits_{s=t^{0}}^{t}\nabla f_{j}(\textbf{x}_{j}^{s})\big\|^{2}+\big\|\beta(t-t^{0})(\textbf{x}_{j}^{t^{0}}-\textbf{v}_{j}^{t^{0}})\big\|^{2}\\
&-2 \alpha(t-t^{0})\sum\limits_{s=t^{0}}^{t}\underbrace{\langle\nabla f_{j}(\textbf{x}_{j}^{s}),\beta(\textbf{x}_{j}^{t^{0}}-\textbf{v}_{j}^{t^{0}})\rangle}\limits_{H_{1}}.\\
\end{aligned}
\end{equation}
Note that $\textbf{x}_{j}^{t}=\textbf{v}_{j}^{t}+\beta(\textbf{x}^{t^{0}}_{j}-\textbf{v}^{t^{0}}_{j})$ for $1\leq t-t^{0}\leq E-1$. By the $\mu$-strong convexity of $f_{i}$ based on Assumption \ref{assum1}, we have
\begin{equation}  \label{eq68-1}
\begin{aligned}
H_{1}&=\langle\nabla f_{j}(\textbf{v}_{j}^{s}+\beta(\textbf{x}^{t^{0}}_{j}-\textbf{v}^{t^{0}}_{j})),\beta(\textbf{x}_{j}^{t^{0}}-\textbf{v}_{j}^{t^{0}})\rangle\\
&\geq f_{j}(\textbf{v}_{j}^{s}+\beta(\textbf{x}^{t^{0}}_{j}-\textbf{v}^{t^{0}}_{j}))-f_{j}(\textbf{v}_{j}^{s})+\frac{\mu}{2}\|\beta(\textbf{x}^{t^{0}}_{j}-\textbf{v}^{t^{0}}_{j})\|^{2}\\
&=f_{j}(\textbf{x}_{j}^{s})-f_{j}^{*}+f_{j}^{*}-f_{j}(\textbf{v}_{j}^{s})+\frac{\mu\beta^{2}}{2}\|\textbf{x}^{t^{0}}_{j}-\textbf{v}^{t^{0}}_{j}\|^{2}\\
&\geq \frac{1}{2L}\|\nabla f_{j}(\textbf{x}_{j}^{s})\|^{2}-\frac{1}{2\mu}\|\nabla f_{j}(\textbf{v}_{j}^{s})\|^{2}+\frac{\mu\beta^{2}}{2}\|\textbf{x}^{t^{0}}_{j}-\textbf{v}^{t^{0}}_{j}\|^{2},
\end{aligned}
\end{equation}
which leads to
\begin{equation}  \label{eq68-2}
\begin{aligned}
-H_{1}\leq \frac{G_{max}^{2}}{2\mu}-\frac{G_{min}^{2}}{2L}-\frac{\mu\beta^{2}\theta^{2}}{2},
\end{aligned}
\end{equation}
where the first and second inequalities in \eqref{eq68-1} use Assumption \ref{assum1} and Assumption \ref{assum2}, \eqref{eq68-2} results from Assumption \ref{assum3} and \eqref{eq63-1}. Substituting \eqref{eq68-2} into \eqref{eq68}, we obtain
\begin{equation}  \label{eq68-3}
\begin{aligned}
\big\|\textbf{x}_{j}^{t}-\textbf{x}_{j}^{t^{0}}\big\|^{2}&\leq\alpha^{2}(t-t^{0})^{2}G_{max}^{2}+\beta^{2}(t-t^{0})^{2}\Theta^{2}\\
&\;\;+2\alpha (t-t^{0})^{2}\big(\frac{G_{max}^{2}}{2\mu}-\frac{G_{min}^{2}}{2L}-\frac{\mu\beta^{2}\theta^{2}}{2}\big)\\
&=(t-t^{0})^{2}\Gamma,
\end{aligned}
\end{equation}
where we use Assumption \ref{assum3}  and \eqref{eq63-1} again.

Following the weighted summation way for \eqref{eq66}, we have
\begin{equation}  \label{eq69}
\begin{aligned}
\big\|\overline{\textbf{x}}_{j}^{t}-\overline{\textbf{x}}_{j}^{t^{0}}\big\|^{2}&=\bigg\|\sum\limits_{l=1}^{N}\overline{w}_{jl}(\textbf{x}_{l}^{t}-\textbf{x}_{l}^{t^{0}})\bigg\|^{2}\\
&\leq\sum\limits_{l=1}^{N}\overline{w}_{jl}\|\textbf{x}_{l}^{t}-\textbf{x}_{l}^{t^{0}}\|^{2}\leq(t-t^{0})^{2}\Gamma,
\end{aligned}
\end{equation}
where the first inequality is based on  Jensen inequality and the second  inequality from the previous result \eqref{eq68-3}.

Combining $\textbf{x}_{j}^{t^{0}}=\overline{\textbf{x}}_{j}^{t^{0}}$, we further obtain
\begin{equation}  \label{eq70}
\begin{aligned}
\big\|\overline{\textbf{x}}_{j}^{t}-\textbf{x}_{j}^{t^{0}}\big\|^{2}\leq(t-t^{0})^{2}\Gamma.
\end{aligned}
\end{equation}
Substituting \eqref{eq68-3} and \eqref{eq70} into the squared form of \eqref{eq65} completes the proof.

\end{proof}

We use the same analysis method,, which is similar to that used in the previous Lemma \ref{lem3}, to immediately obtain the following result.
\begin{lemma}\label{lem6}\textbf{(Bounded disagreement $\|\textbf{x}_{j}^{t}-\overline{\textbf{x}}_{i}^{t}\|$)}
Under Assumption \ref{assum4}, for the exact MUSIC \eqref{eq49}-\eqref{eq50}, it follows that
\begin{equation}  \label{eq71}
\begin{aligned}
\|\textbf{x}_{j}^{t}-\overline{\textbf{x}}_{i}^{t}\|^{2}\leq(4(t-t^{0})\sqrt{\Gamma}+\varepsilon)^{2}
\end{aligned}
\end{equation}
for $0\leq t-t^{0}\leq E-1$.
\end{lemma}

 We now provide a convergence result for the exact MUSIC.

\begin{theorem}\label{them2} Let Assumptions \ref{assum1}-\ref{assum4} and $\alpha\leq\frac{1}{2L}$ hold. If $E$ and $\beta$ satisfy \eqref{eq63-33}, then the exact MUSIC \eqref{eq49}-\eqref{eq50} converges linearly in the mean-square sense to a neighborhood of the optimum solution:
\begin{equation}  \label{eq72}
\begin{aligned}
\big\|\overline{\textbf{x}}_{i}^{kE}-\textbf{x}^{*}\big\|^{2}\leq(1-\mu\alpha)^{kE}\big\|\overline{\textbf{x}}_{i}^{0}-\textbf{x}^{*}\big\|^{2}+D_{3}
\end{aligned}
\end{equation}
for $k=1,2,\ldots,\lfloor T/E\rfloor$, where $E$ satisfies \eqref{eq63-33},
\begin{equation}  \label{eq72-1}
\begin{aligned}
D_{3}&=\frac{(1-(1-\mu\alpha)^{kE})}{1-(1-\mu\alpha)^{E}}\sum\limits_{s=0}^{E-1}(\zeta^{E-1-s}-\theta^{2})(1-\mu\alpha)^{s}\\
&\underrightarrow{k\rightarrow\infty}\;\mathcal{O}\bigg(\frac{(E-1)^{2}\Gamma(16+4\gamma)+2\alpha\tau-\theta^{2}}{\mu\alpha}\bigg),
\end{aligned}
\end{equation}
$\zeta^{s}=(4s\sqrt{\Gamma}+\varepsilon)^{2}+4\gamma s^{2}\Gamma+2\alpha\tau$.
\end{theorem}

\begin{proof}

For any iteration $t$ in the exact MUSIC \eqref{eq49}-\eqref{eq50}, no matter whether $t\in\mathcal{I}_{E}$ or $t\notin\mathcal{I}_{E}$, it is true that $\overline{\textbf{x}}_{i}^{t}=\overline{\textbf{v}}_{i}^{t}+\overline{\textbf{x}}_{i}^{t^{0}}-\overline{\textbf{v}}_{i}^{t^{0}}$, which yields
\begin{equation}  \label{eq73}
\begin{aligned}
\|\overline{\textbf{v}}_{i}^{t}-\textbf{x}^{*}\|\geq\|\overline{\textbf{x}}_{i}^{t}-\textbf{x}^{*}\|-\|\overline{\textbf{x}}_{i}^{t^{0}}-\overline{\textbf{v}}_{i}^{t^{0}}\|\\
\end{aligned}
\end{equation}
and
\begin{equation}  \label{eq73-1}
\begin{aligned}
\|\overline{\textbf{v}}_{i}^{t}-\textbf{x}^{*}\|^{2}\geq&\|\overline{\textbf{x}}_{i}^{t}-\textbf{x}^{*}\|^{2}+\|\overline{\textbf{x}}_{i}^{t^{0}}-\overline{\textbf{v}}_{i}^{t^{0}}\|^{2}\\
&-2\|\overline{\textbf{x}}_{i}^{t}-\textbf{x}^{*}\|\|\overline{\textbf{x}}_{i}^{t^{0}}-\overline{\textbf{v}}_{i}^{t^{0}}\|\\
\geq&\|\overline{\textbf{x}}_{i}^{t}-\textbf{x}^{*}\|^{2}+\|\overline{\textbf{x}}_{i}^{t^{0}}-\overline{\textbf{v}}_{i}^{t^{0}}\|^{2}\\
\geq&\|\overline{\textbf{x}}_{i}^{t}-\textbf{x}^{*}\|^{2}+\theta^{2},
\end{aligned}
\end{equation}
where the inequality \eqref{eq63-1} is used.

Thus, following the previous result \eqref{eq13} on one step gradient descent and combining Lemmas \ref{lem5} and \ref{lem6},  for $0\leq t-t^{0}\leq E-1$, we have
\begin{equation}  \label{eq74}
\begin{aligned}
\parallel\overline{\textbf{v}}_{i}^{t+1}-\textbf{x}^{*}\parallel^{2}\leq(1-\mu\alpha)\|\overline{\textbf{x}}_{i}^{t}-\textbf{x}^{*}\|^{2}+ \zeta^{t-t^{0}}.
\end{aligned}
\end{equation}
Substituting \eqref{eq73-1} into above inequality leads to
\begin{equation}  \label{eq76}
\begin{aligned}
\Delta^{t+1}\leq (1-\mu\alpha)\Delta^{t}+\zeta^{t-t^{0}}-\theta^{2},
\end{aligned}
\end{equation}
where we use $\Delta^{t+1}=\parallel\overline{\textbf{x}}_{i}^{t+1}-\textbf{x}^{*}\parallel^{2}$ for convenience.

Iterating \eqref{eq76} and summing up from $t^{0}+1$ to $t^{0}+E$, we get
\begin{equation}  \label{eq77}
\begin{aligned}
\Delta^{t^{0}+E}&\leq(1-\mu\alpha)^{E}\Delta^{t^{0}}+\sum\limits_{s=0}^{E-1}(\zeta^{E-s-1}-\theta^{2})(1-\mu\alpha)^{s},
\end{aligned}
\end{equation}
or, equivalently,
\begin{equation}  \label{eq78}
\begin{aligned}
\big\|\overline{\textbf{x}}_{i}^{kE}-\textbf{x}^{*}\big\|^{2}\leq(1-\mu\alpha)^{E}\big\|\overline{\textbf{x}}_{i}^{(k-1)E}-\textbf{x}^{*}\big\|^{2}+D_{4},
\end{aligned}
\end{equation}
where we use $D_{4}=\sum\limits_{s=0}^{E-1}(\zeta^{E-s-1}-\theta^{2})(1-\mu\alpha)^{s}$ and $k=1,2,\ldots,\lfloor T/E\rfloor$.
Recursive application of the above relation for $k$ times yields
\begin{equation}  \label{eq79}
\begin{aligned}
\big\|\overline{\textbf{x}}_{i}^{kE}-\textbf{x}^{*}\big\|^{2}\leq(1-\mu\alpha)^{kE}\big\|\overline{\textbf{x}}_{i}^{0}-\textbf{x}^{*}\big\|^{2}+D_{3},
\end{aligned}
\end{equation}
where
\begin{equation}  \label{eq80}
\begin{aligned}
D_{3}=\frac{D_{4}(1-(1-\mu\alpha)^{kE})}{1-(1-\mu\alpha)^{E}}.
\end{aligned}
\end{equation}

For sufficiently large $k$, consider $\varepsilon=0$ and a simple case of $E-s-1= E-1$ for $\zeta^{E-s-1}$ in \eqref{eq78} since $\zeta^{E-s-1}$ is monotone increasing, thus, we can write
\begin{equation}  \label{eq81}
\begin{aligned}
\limsup\limits_{k\rightarrow\infty}\zeta^{E-1}=(E-1)^{2}\Gamma(16+4\gamma)+2\alpha\tau,
\end{aligned}
\end{equation}

\begin{equation}  \label{eq82}
\begin{aligned}
\limsup\limits_{k\rightarrow\infty}D_{4}=\frac{1-(1-\mu\alpha)^{E}}{\mu\alpha}(\limsup\limits_{k\rightarrow\infty}\zeta^{E-1}-\theta^{2}),
\end{aligned}
\end{equation}
and
\begin{equation}  \label{eq83}
\begin{aligned}
\limsup\limits_{k\rightarrow\infty}D_{3}=\frac{(E-1)^{2}\Gamma(16+4\gamma)+2\alpha\tau-\theta^{2}}{\mu\alpha}.
\end{aligned}
\end{equation}

\end{proof}

Theorem \ref{them2} shows that exact MUSIC converges linearly to a steady state point as $k\rightarrow\infty$ regardless of network topology since our analysis does not depend on the condition number of network, which is regarded as a parameter affecting convergence in other literature. Furthermore, by defining
\begin{equation}  \label{eq83-1}
\begin{aligned}
\Upsilon^{kE}=\big\|\overline{\textbf{x}}_{i}^{kE}-\textbf{x}^{*}\big\|^{2}-\frac{\zeta^{E-1}-\theta^{2}}{\mu\alpha}
\end{aligned}
\end{equation}
in \eqref{eq72}, an R-Linear convergence rate can be obtained immediately as follows.

\begin{corollary}\label{corollary2} Under the notations and the conditions of Theorem \ref{them2}, the iterations generated by \eqref{eq83-1} converge R-linearly with
\begin{equation}  \label{eq83-2}
\begin{aligned}
\Upsilon^{kE}\leq(1-\mu\alpha)^{kE}\Upsilon^{0}
\end{aligned}
\end{equation}
for all $k=1,2,\ldots,\lfloor T/E\rfloor$.
\end{corollary}

\subsection{Discussion}\label{IIIB}
\subsubsection{Asymptotic error bound}

From \eqref{eq72-1}, the asymptotic error bound for the proposed exact MUSIC algorithm can be split into three terms
\begin{equation}  \label{eq84}
\begin{aligned}
\mathcal{O}\bigg(\underbrace{\frac{(E-1)^{2}\Gamma(16+4\gamma)}{\mu\alpha}}\limits_{\textrm{local\;drift}}+\underbrace{\frac{2\tau}{\mu}}\limits_{\textrm{inexact\;bias}}-\underbrace{\frac{\theta^{2}}{\mu\alpha}}\limits_{\textrm{bias\;correction}}\bigg).
\end{aligned}
\end{equation}
The first term indicates the local drift caused by performing multiple local updates with insufficient corrections, the second term is a constant inexact bias  independent of $E$ and $\alpha$, which is generated by the inexact ATC diffusion method as in \eqref{eq35-1}  due to the existence of  different local optimums at different agents, the third term is a bias correction, which is used to eliminate the influence  of previous local drift and inexact bias. When $E=1$, we obtain the asymptotic error of exact diffusion as
\begin{equation}  \label{eq85}
\begin{aligned}
\mathcal{O}\bigg(\underbrace{\frac{2\tau}{\mu}}\limits_{\textrm{inexact\;bias}}-\underbrace{\frac{\theta^{2}}{\mu\alpha}}\limits_{\textrm{bias\;correction}}\bigg).
\end{aligned}
\end{equation}
Expression \eqref{eq85} reinterprets the intrinsic mechanism of original exact diffusion to improve the convergence performance by performing bias correction. One can also see that an appropriately smaller $\alpha$ will trigger better error compensation. Such interpretation is different from the one presented in the previous works \cite{yuan2018exact,yuan2018exact2}. In comparison, our exact MUSIC inevitably leads to the local drift in order to enhance convergence rate. Therefore, it is possible that there is a trade-off between the convergence rate and  the required steady state accuracy.

\subsubsection{Necessity of local correction}
When we write $\Gamma=(\alpha^{2}G_{max}^{2}+\Theta^{2}-\mu\alpha\beta^{2}\theta^{2})+\alpha(\frac{ G_{max}^{2}}{\mu}-\frac{ G_{min}^{2}}{L})$, it follows that $\Gamma>\alpha^{2}G_{max}^{2}$ due to the facts of $\frac{G_{min}^{2}}{G_{max}^{2}}<1\leq \kappa$ and $\Theta^{2}>\mu\alpha\beta^{2}\theta^{2}$, where $\kappa\triangleq\frac{L}{\mu}$ is known as the condition number of the function $f_{i}$. Hence, comparing with inexact MUSIC, it is inevitable that exact MUSIC has a larger local drift term. However, local correction is indispensable to exact MUSIC. As a matter of fact, we can design a new algorithm without local correction as follows

\begin{align}
\label{eq86} &\textbf{v}^{t+1}_{i}=\textbf{x}^{t}_{i}-\alpha \nabla f_{i}(\textbf{x}^{t}_{i}),\\
\label{eq87} &\textbf{x}_{i}^{t+1}=\begin{cases}
\textbf{v}^{t+1}_{i} &\textrm{if}\;\; t+1\notin \mathcal{I}_{E} \\
\sum\limits_{j\in \mathcal{N}_{i}}\overline{w}_{ij}(\textbf{v}^{t+1}_{j}+\beta(\textbf{x}^{t^{0}}_{j}-\textbf{v}^{t^{0}}_{j}))  &\textrm{if}\;\; t+1\in \mathcal{I}_{E}
\end{cases},
\end{align}
which performs the bias correction only at the combination step.  Algorithm \eqref{eq86}-\eqref{eq87}  can be regarded as  a distributed version of the EASGD \cite{wang2021cooperative} or an intermediate stage combining inexact MUSIC and exact diffusion.

From the proof of Theorems \ref{them1} and \ref{them2}, we can obtain the following recursive inequations for algorithm \eqref{eq86}-\eqref{eq87}
\begin{equation}  \label{eq88}
\begin{aligned}
&\Delta^{t^{0}+1}\leq (1-\mu\alpha)\Delta^{t^{0}}+\xi^{0},\\
&\Delta^{t^{0}+2}\leq (1-\mu\alpha)\Delta^{t^{0}+1}+\xi^{1},\\
&\;\;\;\;\;\;\;\;\;\;\vdots\\
&\Delta^{t^{0}+E-1}\leq (1-\mu\alpha)\Delta^{t^{0}+E-2}+\xi^{E-2},\\
&\Delta^{t^{0}+E}\leq (1-\mu\alpha)\Delta^{t^{0}+E-1}+\xi^{E-1}-\theta^{2},\\
\end{aligned}
\end{equation}
where we use the same definitions for $\Delta$, $\xi$ and $\theta$ given in Theorems \ref{them1} and \ref{them2}, respectively. Consequently,  by following a similar approach as in Theorem \ref{them1}, we omit the proofs  for brevity and obtain roughly the steady-state error
\begin{equation}  \label{eq89}
\begin{aligned}
\mathcal{O}\bigg(\underbrace{\frac{(E-1)^{2} (16+4\gamma)\alpha G_{max}^{2}}{\mu}}\limits_{\textrm{local\;drift}}+\underbrace{\frac{2\tau}{\mu}}\limits_{\textrm{inexact\;bias}}-\underbrace{\frac{\theta^{2}}{1-(1-\mu\alpha)^{E}}}\limits_{\textrm{bias\;correction}}\bigg).
\end{aligned}
\end{equation}

From \eqref{eq89}, the new algorithm has the same local drift and inexact bias as the inexact MUSIC but smaller bias correction than the exact MUSIC due to $1-(1-\mu\alpha)^{E}>\mu\alpha$. With $E>1$, it is therefore that local bias correction is necessary in order to achieve good exact solution.

\subsubsection{Communication complexity}

Let $T_{\epsilon}$ denote the number of required iteration steps for MUSIC to achieve an $\epsilon$ accuracy level. From \eqref{eq72}, it follows that the number of required communication rounds to achieve the target accuracy of $\epsilon$ is $\frac{T_{\epsilon}}{E}=\mathcal{O}(\frac{2\kappa}{E}\log (\frac{1}{\epsilon}))$, which is
reduced by a factor of $\frac{1}{E}$ over $\mathcal{O}(2\kappa\log (\frac{1}{\epsilon}))$ achieved by exact diffusion and better than many exiting algorithms, such as NIDS, AugDGM, NEXT, DIGing. Correspondingly, exact MUSIC has the complexity of gradient evaluation of $\mathcal{O}(2\kappa\log (\frac{1}{\epsilon}))$ due to $E$ local updates per one communication round. The complexity comparison with existing state-of-the-art methods are presented in Table \ref{tab1}, which verifies in theory that our exact MUSIC is communication efficient. Moreover, it should be noted that our topology-independent complexity analysis only requires the connected network without the restriction of specified network topologies.

\subsubsection{Choices of $E$ and $\beta$}  From Corollary 1, $E$ and $\beta$ are required to satisfy the inequality \eqref{eq63-33} to ensure convergence. According to \eqref{eq63-33}, when $\beta=0$ (indicating exact diffusion), we obtain that  $E\geq0$. When $\beta>0$, we can rewrite \eqref{eq63-33} as $E\geq \log_{\nu}\big(1-\frac{\|\textbf{Z}\|}{\frac{1}{\beta}-\frac{\nu}{1-\nu}\|\textbf{Z}-\textbf{I}\|}\big)$ under $\beta<\min\{\frac{1-\nu}{\nu}\|\textbf{Z}-\textbf{I}\|,1\}$. This implies that a large value of $E$ can be selected as long as $\beta$ is sufficiently small. However, it is well-established that a large $E$  can lead to substantial local drift.  Thus, a tradeoff exists between the variables $E$ and $\beta$.  In practice, we select manually the size of $E$ as in the inexact case (e.g., 2, 3, 4) and a large $\beta$ approaching to 1. This selection for a small $E$ is made due to its evident acceleration effect.

\subsection{Numerical Results}\label{IIIC}

\subsubsection{Distributed least squares problem}

We first perform experimental comparison on solving the same least squares problem given as before in \eqref{eq43}. In this subsection, in addition to compare exact MUSIC with the original exact diffusion, we also compare performance with  the linearly convergent algorithms, such as EXTRA \cite{shi2015extra}, DIGing \cite{nedic2017achieving} and three state-of-the-art accelerated benchmarks including ACC-EXTRA \cite{li2020revisiting}, ACC-GT \cite{li2021accelerated} and Acc-DNGD-SC \cite{8812696}. The experimental setup is the same with previous section, except for the step size $\alpha=0.002$ for exact MUSIC and EXTRA.  All other parameters required in accelerated algorithms are hand-optimized to achieve the best performance. In addition, we  test the performance by trying to set $\beta=1$. Note that the problem in this example is ill-conditioned with large condition number by setting $\mu=10^{-6}$ in order to illustrate the algorithmic advantages.

In Fig. \ref{fig4}, one can see that our exact MUSIC converges linearly to the exact solution and achieves an equivalent steady state error as the exact diffusion but with less communication and faster rate. It should be noted that exact MUSIC performs well for $1\leq E\leq4$, but when $E\geq5$ significant divergences are observed. This is mainly because too large quantity $E$ results in the failure of boundedness of bias correction $\|\overline{\textbf{x}}_{i}^{t^{0}}-\overline{\textbf{v}}_{i}^{t^{0}}\|$.  Comparing with the best performance obtained by ACC-EXTRA among those accelerated exact algorithms, our exact MUSIC with limited local iterations (e.g., $E=2,3,4$) shows almost identical steady state accuracy but with fastest convergence rate.

\begin{figure}[htbp]
\centering
\centerline{\includegraphics[width=7cm,height=4.5cm]{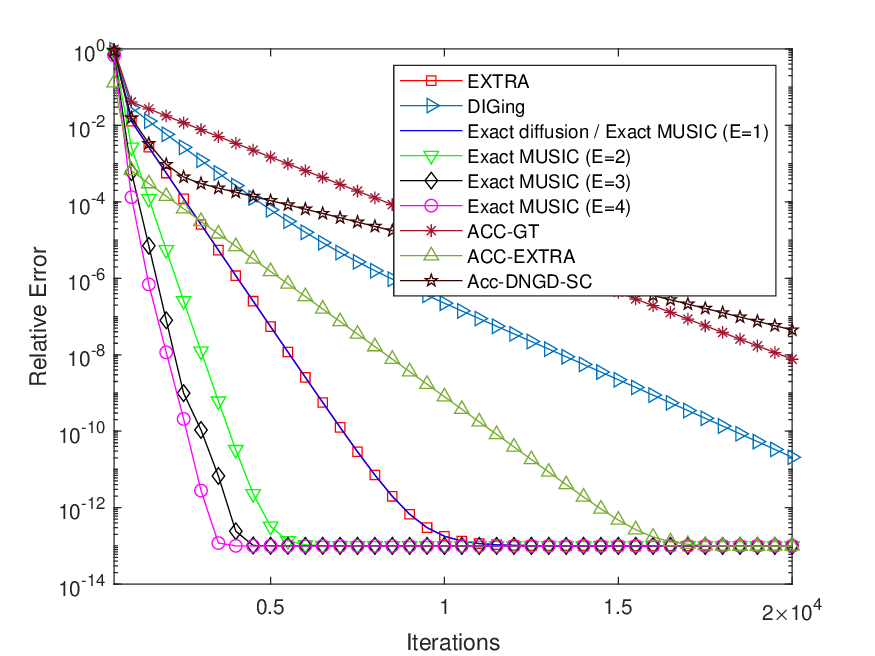}}
\caption{Performance comparisons measured in terms of relative error with respect to iterations on the least squares problem based on synthesized data.}\label{fig4}
\end{figure}

\textbf{Real dataset.}  We report  the results obtained from the ``letter'' dataset provided by  LIBSVM \cite{CC01a}. Here, we selected $10^{4}$ training samples to generate the matrix $\textbf{A}_{i}\in\mathbb{R}^{p\times m}$ with $p=16, N=m=100$ and the vector $b_{i}\in\mathbb{R}^{m}$ corresponding to 26 possible letter labels. The other parameters are consistent with those used in the synthetic data experiment. Similar  to  the synthetic dataset case, Fig. \ref{fig5} shows that our exact MUSIC achieves the best overall performance regardless of $E=2,3$  or 4. Meanwhile, it is also observed that multiple local updates have a negligible impact  on the  final steady state error for solving this least squares problem.

\begin{figure}[htbp]
\centering
\centerline{\includegraphics[width=7cm,height=4.5cm]{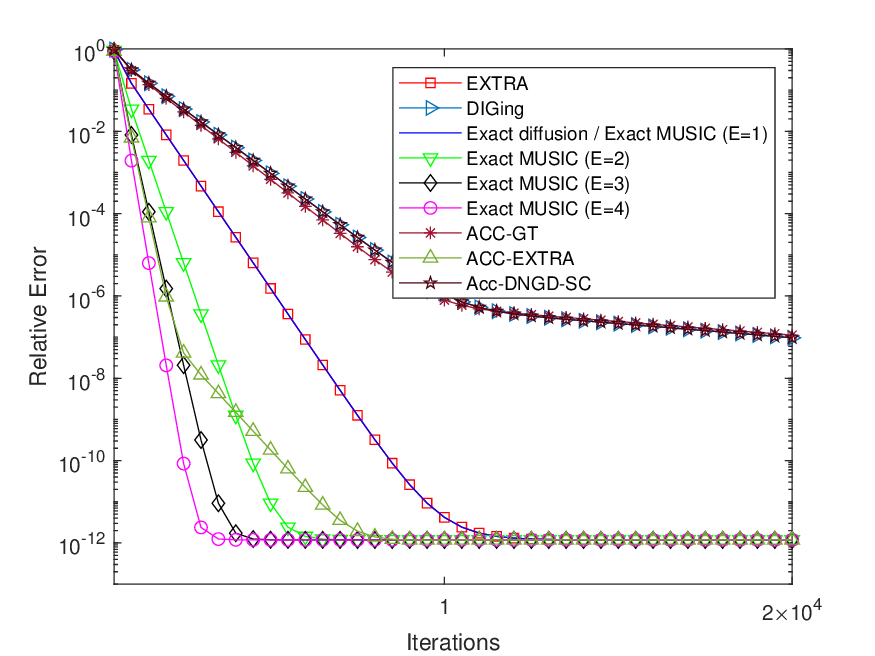}}
\caption{Performance comparisons measured in terms of relative error with respect to iterations on the least squares problem based on real dataset.}\label{fig5}
\end{figure}

\subsubsection{Distributed Logistic Regression}

In this subsection, we test the performance of exact MUSIC by solving a representative  logistic regression learning problem for binary classification, where each agent is associated with a local cost function
\begin{equation}  \label{eq90}
f_{i}(\textbf{x})=\frac{1}{m}\sum\limits_{j=1}^{m}\ln(1+\exp(-\gamma_{i,j}\textbf{h}_{i,j}^{T}\textbf{x}))+\frac{\mu}{2}\|\textbf{x}\|^{2}.
\end{equation}
Here $\{\textbf{h}_{i,j}\in \mathbb{R}^{p}\}$ is the feature vector, and $\gamma_{i,j}\in\{-1,1\}$ is the corresponding label. We still use the ``letter'' dataset and the Erdos-Renyi model with average degree 4 to generate a connected network. We split the ``letter'' datasubset using the second and fourth labels to $N=50$ agents, where each agent receives $m = 30$ training samples of dimension $p=16$. In this problem, since the optimal $\textbf{x}^{*}$ is unknown, we approximate it by running the centralized gradient descent with a very small step size for $2\times10^{5}$ iterations.

\begin{figure}[htbp]
\centering
\centerline{\includegraphics[width=7cm,height=4.5cm]{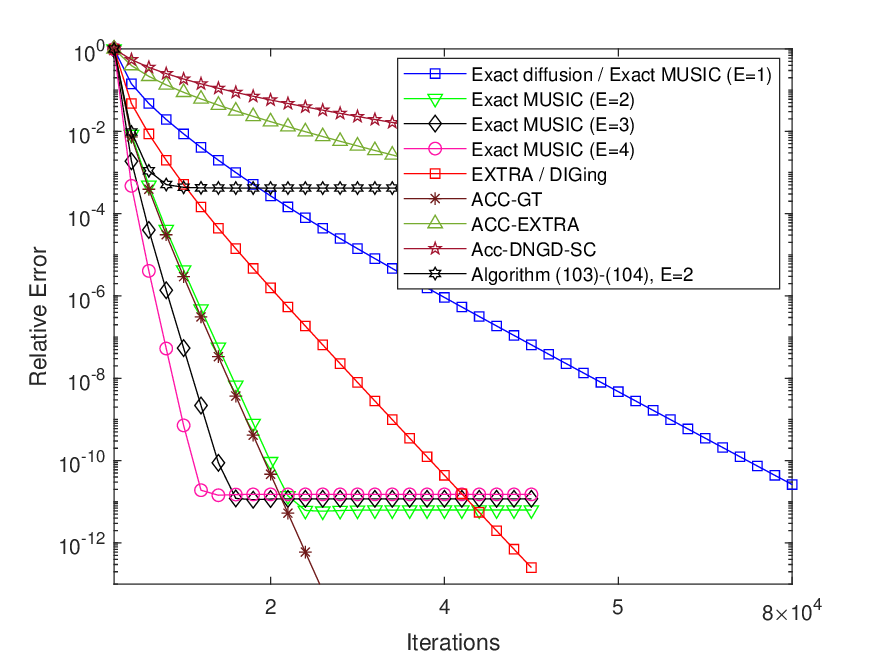}}
\caption{Performance comparisons measured in terms of relative error with respect to iterations on the logistic regression problem based on real dataset.}\label{fig6}
\end{figure}

\begin{figure}[htbp]
\centering
\centerline{\includegraphics[width=7cm,height=4.5cm]{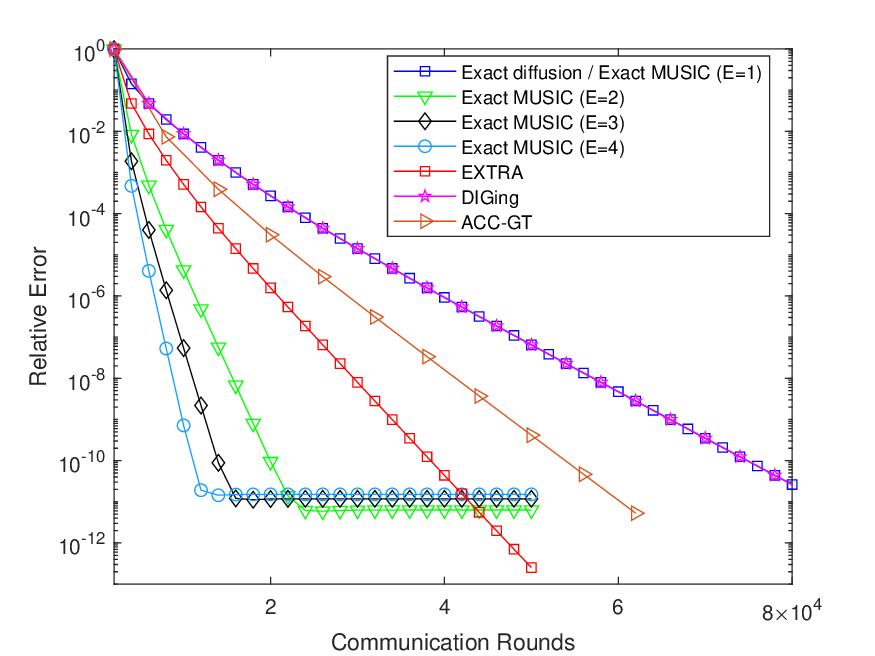}}
\caption{Performance comparisons measured in terms of relative error with respect to number of communications on the logistic regression problem based on real dataset.}\label{fig7}
\end{figure}

From the results shown in Fig. \ref{fig6}, one can see that exact MUSIC enhances significantly converge rate of exact diffusion as the increase of $E$ under the same step size. On the other hand, exact MUSIC achieve the level of high accuracy of $10^{-11}$, which is enough to satisfy the accuracy requirement for the vast majority of learning applications. Without local correction, algorithm \eqref{eq86}-\eqref{eq87} can not converge to a highly accurate solution in spite of the same fast convergence rate as exact MUSIC. This verifies the necessary of local correction as explained theoretically in section \ref{IIIB}.

While the ACC-GT algorithm demonstrates a convergence rate comparable to that of our exact MUSIC and offers higher estimation accuracy, a notable degradation in convergence rate becomes apparent when considering communication costs, as depicted in Fig. \ref{fig7}. In other words, ACC-GT necessitates a greater number of communication rounds than exact MUSIC to achieve the same level of accuracy. This discrepancy primarily arises from the fact that ACC-GT and the DIGing methods require three and two communication rounds per iteration, respectively, whereas our exact MUSIC, along with exact diffusion/EXTRA, only requires one. This observation also underscores the advantage of communication efficiency of our method, which is also verified by theoretically communication complexity. Fig. \ref{fig7} does not show the performance curves of ACC-EXTRA and Acc-DNGD-SC due to  lack of competitiveness  in this example.

Overall, our exact MUSIC proves to be well-suited for a wide range of distributed optimization problems, as it simultaneously targets three key objectives: rapid convergence, efficient communication, and competitive accuracy in exact solutions.

\section{Conclusion and future work}\label{V}

In this paper, we propose an accelerated framework termed the Multi-Updates SIngle-Combination (MUSIC) for first-order distributed optimization.  To our knowledge, MUSIC is the first multiple local updates scheme in deterministic rather than stochastic settings, which can provide a visible acceleration with less communication complexity. To apply MUSIC, we first design the inexact MUSIC method that deploys the traditional ATC method into this framework. Following the success of inexact MUSIC in terms of convergence rate and accuracy, we further develop the exact MUSIC, which has a very different strategy compared with inexact MUSIC. In addition to multiple updates, exact MUSIC employs multiple local bias corrections, thereby converging to the exact solution. Our detailed convergence analysis on inexact and exact MUSIC methods provides the guarantee of linear convergence under mild conditions and the decrease of communication complexity. Future work will focus on further improvement on estimate accuracy under MUSIC framework and the feasibility to develop MUSIC-based second-order methods.

\cite{liu2023decentralized}

\bibliographystyle{IEEEtran}
\bibliography{mybibfile}

\begin{thebibliography}{10}
\providecommand{\url}[1]{#1}
\csname url@samestyle\endcsname
\providecommand{\newblock}{\relax}
\providecommand{\bibinfo}[2]{#2}
\providecommand{\BIBentrySTDinterwordspacing}{\spaceskip=0pt\relax}
\providecommand{\BIBentryALTinterwordstretchfactor}{4}
\providecommand{\BIBentryALTinterwordspacing}{\spaceskip=\fontdimen2\font plus
\BIBentryALTinterwordstretchfactor\fontdimen3\font minus
  \fontdimen4\font\relax}
\providecommand{\BIBforeignlanguage}[2]{{%
\expandafter\ifx\csname l@#1\endcsname\relax
\typeout{** WARNING: IEEEtran.bst: No hyphenation pattern has been}%
\typeout{** loaded for the language `#1'. Using the pattern for}%
\typeout{** the default language instead.}%
\else
\language=\csname l@#1\endcsname
\fi
#2}}
\providecommand{\BIBdecl}{\relax}
\BIBdecl

\bibitem{jaleel2020distributed}
H.~Jaleel and J.~S. Shamma, ``Distributed optimization for robot networks: From
  real-time convex optimization to game-theoretic self-organization,''
  \emph{Proceedings of the IEEE}, vol. 108, no.~11, pp. 1953--1967, 2020.

\bibitem{molzahn2017survey}
D.~K. Molzahn, F.~D{\"o}rfler, H.~Sandberg, S.~H. Low, S.~Chakrabarti,
  R.~Baldick, and J.~Lavaei, ``A survey of distributed optimization and control
  algorithms for electric power systems,'' \emph{IEEE Transactions on Smart
  Grid}, vol.~8, no.~6, pp. 2941--2962, 2017.

\bibitem{nedic2020distributed}
A.~Nedic, ``Distributed gradient methods for convex machine learning problems
  in networks,'' \emph{IEEE Signal Processing Magazine}, vol.~37, no.~3, pp.
  92--101, 2020.

\bibitem{yang2017collaborative}
S.~Yang, Q.~Liu, and J.~Wang, ``A collaborative neurodynamic approach to
  multiple-objective distributed optimization,'' \emph{IEEE Transactions on
  Neural Networks and Learning Systems}, vol.~29, no.~4, pp. 981--992, 2017.

\bibitem{nedic2009distributed}
A.~Nedic and A.~Ozdaglar, ``Distributed subgradient methods for multi-agent
  optimization,'' \emph{IEEE Transactions on Automatic Control}, vol.~54,
  no.~1, pp. 48--61, 2009.

\bibitem{nedic2016stochastic}
A.~Nedi{\'c} and A.~Olshevsky, ``Stochastic gradient-push for strongly convex
  functions on time-varying directed graphs,'' \emph{IEEE Transactions on
  Automatic Control}, vol.~61, no.~12, pp. 3936--3947, 2016.

\bibitem{7963560}
A.~Nedi\'{c}, A.~Olshevsky, W.~Shi, and C.~A. Uribe, ``Geometrically convergent
  distributed optimization with uncoordinated step-sizes,'' in \emph{2017
  American Control Conference (ACC)}, 2017, pp. 3950--3955.

\bibitem{li2020distributed}
H.~Li, H.~Cheng, Z.~Wang, and G.-C. Wu, ``Distributed nesterov gradient and
  heavy-ball double accelerated asynchronous optimization,'' \emph{IEEE
  Transactions on Neural Networks and Learning Systems}, vol.~32, no.~12, pp.
  5723--5737, 2020.

\bibitem{lu2020nesterov}
Q.~L{\"u}, X.~Liao, H.~Li, and T.~Huang, ``A \textsc{n}esterov-like gradient
  tracking algorithm for distributed optimization over directed networks,''
  \emph{IEEE Transactions on Systems, Man, and Cybernetics: Systems}, vol.~51,
  no.~10, pp. 6258--6270, 2021.

\bibitem{koloskova2021improved}
A.~Koloskova, T.~Lin, and S.~U. Stich, ``An improved analysis of gradient
  tracking for decentralized machine learning,'' \emph{Advances in Neural
  Information Processing Systems}, vol.~34, pp. 11\,422--11\,435, 2021.

\bibitem{pu2021distributed}
S.~Pu and A.~Nedi{\'c}, ``Distributed stochastic gradient tracking methods,''
  \emph{Mathematical Programming}, vol. 187, no.~1, pp. 409--457, 2021.

\bibitem{9773308}
J.~Liu, Z.~Yu, and D.~W.~C. Ho, ``Distributed constrained optimization with
  delayed subgradient information over time-varying network under adaptive
  quantization,'' \emph{IEEE Transactions on Neural Networks and Learning
  Systems}, pp. 1--14, Early Access, 2022.

\bibitem{kairouz2021advances}
P.~Kairouz, H.~B. McMahan, B.~Avent, A.~Bellet, M.~Bennis, A.~N. Bhagoji,
  K.~Bonawitz, Z.~Charles, G.~Cormode, R.~Cummings \emph{et~al.}, ``Advances
  and open problems in federated learning,'' \emph{Foundations and Trends in
  Machine Learning}, vol.~14, no. 1--2, pp. 1--210, 2021.

\bibitem{yang2019federated}
Q.~Yang, Y.~Liu, Y.~Cheng, Y.~Kang, T.~Chen, and H.~Yu, ``Federated learning,''
  \emph{Synthesis Lectures on Artificial Intelligence and Machine Learning},
  vol.~13, no.~3, pp. 1--207, 2019.

\bibitem{konevcny2016federated}
J.~Kone{\v{c}}n{\`y}, H.~B. McMahan, F.~X. Yu, P.~Richt{\'a}rik, A.~T. Suresh,
  and D.~Bacon, ``Federated learning: Strategies for improving communication
  efficiency,'' in \emph{NIPS Workshop on Private Multi-Party Machine
  Learning}, 2016.

\bibitem{li2020federated}
T.~Li, A.~K. Sahu, A.~Talwalkar, and V.~Smith, ``Federated learning:
  Challenges, methods, and future directions,'' \emph{IEEE Signal Processing
  Magazine}, vol.~37, no.~3, pp. 50--60, 2020.

\bibitem{yuan2016convergence}
K.~Yuan, Q.~Ling, and W.~Yin, ``On the convergence of decentralized gradient
  descent,'' \emph{SIAM Journal on Optimization}, vol.~26, no.~3, pp.
  1835--1854, 2016.

\bibitem{nedic2014distributed}
A.~Nedi{\'c} and A.~Olshevsky, ``Distributed optimization over time-varying
  directed graphs,'' \emph{IEEE Transactions on Automatic Control}, vol.~60,
  no.~3, pp. 601--615, 2014.

\bibitem{sayed2014adaptation}
A.~H. Sayed, ``Adaptation, learning, and optimization over networks,''
  \emph{Foundations and Trends in Machine Learning}, vol.~7, no. 4-5, pp.
  311--801, 2014.

\bibitem{sayed2014adaptive}
------, ``Adaptive networks,'' \emph{Proceedings of the IEEE}, vol. 102, no.~4,
  pp. 460--497, 2014.

\bibitem{5585721}
A.~Nedic, ``Asynchronous broadcast-based convex optimization over a network,''
  \emph{IEEE Transactions on Automatic Control}, vol.~56, no.~6, pp.
  1337--1351, 2011.

\bibitem{sundhar2010distributed}
S.~Sundhar~Ram, A.~Nedi{\'c}, and V.~V. Veeravalli, ``Distributed stochastic
  subgradient projection algorithms for convex optimization,'' \emph{Journal of
  optimization theory and applications}, vol. 147, no.~3, pp. 516--545, 2010.

\bibitem{9406173}
Z.~Li, B.~Liu, and Z.~Ding, ``Consensus-based cooperative algorithms for
  training over distributed data sets using stochastic gradients,'' \emph{IEEE
  Transactions on Neural Networks and Learning Systems}, vol.~33, no.~10, pp.
  5579--5589, 2022.

\bibitem{9286906}
W.~Tao, G.~W. Wu, and Q.~Tao, ``Momentum acceleration in the individual
  convergence of nonsmooth convex optimization with constraints,'' \emph{IEEE
  Transactions on Neural Networks and Learning Systems}, vol.~33, no.~3, pp.
  1107--1118, 2022.

\bibitem{shi2015extra}
W.~Shi, Q.~Ling, G.~Wu, and W.~Yin, ``\textsc{EXTRA}: An exact first-order
  algorithm for decentralized consensus optimization,'' \emph{SIAM Journal on
  Optimization}, vol.~25, no.~2, pp. 944--966, 2015.

\bibitem{li2020revisiting}
H.~Li and Z.~Lin, ``Revisiting extra for smooth distributed optimization,''
  \emph{SIAM Journal on Optimization}, vol.~30, no.~3, pp. 1795--1821, 2020.

\bibitem{9772052}
X.~Jiang, X.~Zeng, J.~Sun, and J.~Chen, ``Distributed stochastic gradient
  tracking algorithm with variance reduction for non-convex optimization,''
  \emph{IEEE Transactions on Neural Networks and Learning Systems}, vol.~34,
  no.~9, pp. 5310--5321, 2023.

\bibitem{li2019decentralized}
Z.~Li, W.~Shi, and M.~Yan, ``A decentralized proximal-gradient method with
  network independent step-sizes and separated convergence rates,'' \emph{IEEE
  Transactions on Signal Processing}, vol.~67, no.~17, pp. 4494--4506, 2019.

\bibitem{sun2022distributed}
Y.~Sun, G.~Scutari, and A.~Daneshmand, ``Distributed optimization based on
  gradient tracking revisited: Enhancing convergence rate via surrogation,''
  \emph{SIAM Journal on Optimization}, vol.~32, no.~2, pp. 354--385, 2022.

\bibitem{li2020communication}
B.~Li, S.~Cen, Y.~Chen, and Y.~Chi, ``Communication-efficient distributed
  optimization in networks with gradient tracking and variance reduction,'' in
  \emph{International Conference on Artificial Intelligence and
  Statistics}.\hskip 1em plus 0.5em minus 0.4em\relax PMLR, 2020, pp.
  1662--1672.

\bibitem{nedic2017achieving}
A.~Nedic, A.~Olshevsky, and W.~Shi, ``Achieving geometric convergence for
  distributed optimization over time-varying graphs,'' \emph{SIAM Journal on
  Optimization}, vol.~27, no.~4, pp. 2597--2633, 2017.

\bibitem{li2022consensus}
Z.~Li, B.~Liu, and Z.~Ding, ``Consensus-based cooperative algorithms for
  training over distributed data sets using stochastic gradients,'' \emph{IEEE
  Transactions on Neural Networks and Learning Systems}, vol.~33, pp.
  5579--5589, 2022.

\bibitem{yuan2018exact}
K.~Yuan, B.~Ying, X.~Zhao, and A.~H. Sayed, ``Exact diffusion for distributed
  optimization and learning-part \textsc{I}: Algorithm development,''
  \emph{IEEE Transactions on Signal Processing}, vol.~67, no.~3, pp. 708--723,
  2018.

\bibitem{yuan2018exact2}
------, ``Exact diffusion for distributed optimization and learning-part
  \textsc{II}: Convergence analysis,'' \emph{IEEE Transactions on Signal
  Processing}, vol.~67, no.~3, pp. 724--739, 2018.

\bibitem{yuan2020influence}
K.~Yuan, S.~A. Alghunaim, B.~Ying, and A.~H. Sayed, ``On the influence of
  bias-correction on distributed stochastic optimization,'' \emph{IEEE
  Transactions on Signal Processing}, vol.~68, pp. 4352--4367, 2020.

\bibitem{berahas2018balancing}
A.~S. Berahas, R.~Bollapragada, N.~S. Keskar, and E.~Wei, ``Balancing
  communication and computation in distributed optimization,'' \emph{IEEE
  Transactions on Automatic Control}, vol.~64, no.~8, pp. 3141--3155, 2019.

\bibitem{li2021accelerated}
H.~Li and Z.~Lin, ``Accelerated gradient tracking over time-varying graphs for
  decentralized optimization,'' \emph{arXiv preprint arXiv:2104.02596}, 2021.

\bibitem{8812696}
G.~Qu and N.~Li, ``Accelerated distributed \textsc{N}esterov gradient
  descent,'' \emph{IEEE Transactions on Automatic Control}, vol.~65, no.~6, pp.
  2566--2581, 2020.

\bibitem{kovalev2020optimal}
D.~Kovalev, A.~Salim, and P.~Richt{\'a}rik, ``Optimal and practical algorithms
  for smooth and strongly convex decentralized optimization,'' \emph{Advances
  in Neural Information Processing Systems}, vol.~33, pp. 18\,342--18\,352,
  2020.

\bibitem{mangasarian1995parallel}
L.~Mangasarian, ``Parallel gradient distribution in unconstrained
  optimization,'' \emph{SIAM Journal on Control and Optimization}, vol.~33,
  no.~6, pp. 1916--1925, 1995.

\bibitem{stich2018local}
\BIBentryALTinterwordspacing
S.~U. Stich, ``Local \textsc{SGD} converges fast and communicates little,'' in
  \emph{International Conference on Learning Representations (ICLR)}, 2019.
  [Online]. Available: \url{https://openreview.net/forum?id=S1g2JnRcFX}
\BIBentrySTDinterwordspacing

\bibitem{wang2021cooperative}
J.~Wang and G.~Joshi, ``Cooperative \textsc{SGD}: A unified framework for the
  design and analysis of local-update \textsc{SGD} algorithms,'' \emph{Journal
  of Machine Learning Research}, vol.~22, pp. 1--50, 2021.

\bibitem{khaled2020tighter}
A.~Khaled, K.~Mishchenko, and P.~Richt{\'a}rik, ``Tighter theory for local
  \textsc{SGD} on identical and heterogeneous data,'' in \emph{International
  Conference on Artificial Intelligence and Statistics}.\hskip 1em plus 0.5em
  minus 0.4em\relax PMLR, 2020, pp. 4519--4529.

\bibitem{nesterov2018lectures}
Y.~Nesterov \emph{et~al.}, \emph{Lectures on convex optimization}.\hskip 1em
  plus 0.5em minus 0.4em\relax Springer, 2018, vol. 137.

\bibitem{xi2016distributed}
C.~Xi and U.~A. Khan, ``Distributed subgradient projection algorithm over
  directed graphs,'' \emph{IEEE Transactions on Automatic Control}, vol.~62,
  no.~8, pp. 3986--3992, 2016.

\bibitem{simonetto2017decentralized}
A.~Simonetto, A.~Koppel, A.~Mokhtari, G.~Leus, and A.~Ribeiro, ``Decentralized
  prediction-correction methods for networked time-varying convex
  optimization,'' \emph{IEEE Transactions on Automatic Control}, vol.~62,
  no.~11, pp. 5724--5738, 2017.

\bibitem{simonetto2016class}
A.~Simonetto, A.~Mokhtari, A.~Koppel, G.~Leus, and A.~Ribeiro, ``A class of
  prediction-correction methods for time-varying convex optimization,''
  \emph{IEEE Transactions on Signal Processing}, vol.~64, no.~17, pp.
  4576--4591, 2016.

\bibitem{lobel2010distributed}
I.~Lobel and A.~Ozdaglar, ``Distributed subgradient methods for convex
  optimization over random networks,'' \emph{IEEE Transactions on Automatic
  Control}, vol.~56, no.~6, pp. 1291--1306, 2010.

\bibitem{7402509}
J.~Xu, S.~Zhu, Y.~C. Soh, and L.~Xie, ``Augmented distributed gradient methods
  for multi-agent optimization under uncoordinated constant stepsizes,'' in
  \emph{2015 54th IEEE Conference on Decision and Control (CDC)}, 2015, pp.
  2055--2060.

\bibitem{qu2017harnessing}
G.~Qu and N.~Li, ``Harnessing smoothness to accelerate distributed
  optimization,'' \emph{IEEE Transactions on Control of Network Systems},
  vol.~5, no.~3, pp. 1245--1260, 2017.

\bibitem{alghunaim2020decentralized}
S.~A. Alghunaim, E.~K. Ryu, K.~Yuan, and A.~H. Sayed, ``Decentralized proximal
  gradient algorithms with linear convergence rates,'' \emph{IEEE Transactions
  on Automatic Control}, vol.~66, no.~6, pp. 2787--2794, 2020.

\bibitem{nesterov2003introductory}
Y.~Nesterov, \emph{Introductory lectures on convex optimization: A basic
  course}.\hskip 1em plus 0.5em minus 0.4em\relax Springer Science \& Business
  Media, 2003, vol.~87.

\bibitem{CC01a}
C.~C. Chang and C.~J. Lin, ``{LIBSVM}: A library for support vector machines,''
  \emph{ACM Transactions on Intelligent Systems and Technology}, vol.~2, pp.
  27:1--27:27, 2011, software available at
  \url{http://www.csie.ntu.edu.tw/~cjlin/libsvm}.

\bibitem{liu2023decentralized}
Y.~Liu, T.~Lin, A.~Koloskova, and S.~U. Stich, ``Decentralized gradient
  tracking with local steps,'' \emph{arXiv preprint arXiv:2301.01313}, 2023.

\end{thebibliography}

\end{document}